\documentclass{article}
\usepackage{mathrsfs, euscript,amstext,amsmath,amsthm,amssymb,amscd,a4}

\newtheorem{theorem}{Theorem}[section]
\newtheorem{lemma}[theorem]{Lemma}

\theoremstyle{definition}
\newtheorem{definition}[theorem]{Definition}

\newtheorem{proposition}[theorem]{Proposition}
\theoremstyle{remark}
\newtheorem{remark}[theorem]{Remark}

\numberwithin{equation}{section}

\begin{document}

\title{A Cheeger-M\"{u}ller theorem for symmetric bilinear\\ torsions on manifolds with boundary
}

\author{Guangxiang Su\footnote{Chern Institute of Mathematics, Nankai University, Tianjin 300071, P. R. China. E-mail: guangxiangsu@nankai.edu.cn}}




\date{}

\maketitle

\begin{abstract}
In this paper, we extend Su-Zhang's Cheeger-M\"{u}ller type theorem for symmetric bilinear torsions to manifolds with boundary in the
case that the Riemannian metric and the non-degenerate symmetric bilinear form are of product structure near the boundary. Our result also extends Br\"{u}ning-Ma's Cheeger-M\"{u}ller type theorem for Ray-Singer metric on manifolds with boundary to symmetric bilinear torsions in product case.
We also compare it with the Ray-Singer analytic torsion on manifolds with boundary.
\end{abstract}

\section{Introduction}
Let $F$ be a unitary flat vector bundle on a closed Riemannian manifold $X$. In \cite{RS}, Ray and Singer defined an
analytic torsion associated to $(X,F)$ and proved that it does not depend on the Riemannian metric on $X$. Moreover, they
conjectured that this analytic torsion coincides with the classical Reidemeister torsion defined using a triangulation on
$X$ (cf. \cite{Mi}). This conjecture was later proved in the celebrated papers of Cheeger \cite{C} and M\"{u}ller
\cite{Mu1}. M\"{u}ller generalized this result in \cite{Mu2} to the case where $F$ is a unimodular flat vector bundle on
$X$. In \cite{BZ1}, inspired by the considerations of Quillen \cite{Q}, Bismut and Zhang reformulated the above
Cheeger-M\"{u}ller theorem as an equality between the Reidemeister and Ray-Singer metrics defined on the determinant
of cohomology, and proved an extension of it to the case of general flat vector bundles over $X$. The method used in
\cite{BZ1} is different from those of Cheeger and M\"{u}ller in that it makes use of a deformation by Morse functions
introduced by Witten \cite{W} on the de Rham complex.

On the other hand, if there is a non-degenerate
symmetric bilinear form $b^{F}$ on $F$, Burghelea and Haller \cite{BH1}, defined a complex valued analytic torsion in
the spirit of Ray and Singer \cite{RS}. In \cite{BH1}, they also made an explicit conjecture between the complex valued
analytic torsion and the Turaev torsion (cf. \cite{FT}, \cite{T}). In \cite{SZ}, Su and Zhang used the approach developed by Bismut-Zhang
\cite{BZ1,BZ2}, making use of the Witten deformation, proved the conjecture in full generality. In \cite{BH2},
Burghelea and Haller proved their conjecture, up to sign, in the case where $X$ is of odd dimensional.

Now consider $X$ with the boundary $Y\neq \emptyset$. In \cite{Luc}, \cite{V} and \cite{H}, the authors studied the
Ray-Singer analytic torsion under the assumption that the Hermitian metric $h^{F}$ on $F$ is flat and the
Riemannian metric
$g^{TX}$ has product structure near the boundary. Dai and Fang \cite{DF} studied the case that the Hermitian metric
$h^{F}$ is flat but without assuming a product structure for $g^{TX}$ near $Y$. In \cite{BM1}, Br\"{u}ning and Ma
studied the general case without assumption on $h^{F}$ and $g^{TX}$ associated to the absolute boundary condition on
$X$. In \cite{BM}, Br\"{u}ning and Ma studied the results for the relative boundary condition on $X$, proved a Cheeger-M\"{u}ller type theorem for the Ray-Singer metric on the manifolds with boundary and applied it
derived the gluing formula for the analytic torsion in general setting.

For the complex-valued analytic torsion, Molina \cite{M} extended the Burghelea-Haller analytic torsion to compact manifolds with boundary under the relative
boundary condition and the absolute boundary condition.

In this paper, we extend the main result in \cite{SZ} to manifolds with boundary for the couple $(g^{TX},b^{F})$
are of
product structure near the boundary. We use the method in \cite{BM}. We first double $X$ along the boundary $Y$ and get
a closed Riemannian manifold $\widetilde{X}=X\cup_{Y}X$ with Riemannian metric $g^{T\widetilde{X}}$, then there is a $\mathbb{Z}_{2}$-action on $\widetilde{X}$ induced
by the natural involution $\phi$ on it. Also the flat bundle $F$ extends to a flat bundle $\widetilde{F}$ on $\widetilde{X}$ which has a non-degenerate symmetric bilinear form $b^{\widetilde{F}}$. The metric $g^{T\widetilde{X}}$ and the form $b^{\widetilde{F}}$ are all invariant under the action of $\phi$. So we
first extend the main result in \cite{SZ}
to $\mathbb{Z}_{2}$-equivariant case. Then we compare this $\mathbb{Z}_{2}$-equivariant symmetric
bilinear torsions on $\widetilde{X}$ to the symmetric bilinear torsions on $X$ with relative boundary condition and absolute
boundary condition. Combining these and the result in \cite{SZ}, we get the main theorem in this paper. In a next paper, I will apply the techniques in this paper to deal with the Cappell-Miller analytic torsion \cite{CM}.

The rest of the paper is organized as follows. In Section 2, we review the symmetric bilinear analytic torsion on
manifolds with boundary and the anomaly formula from \cite{M}. In Section 3, we define the symmetric bilinear Milnor torsion on manifolds with boundary.
 In Section 4, we study the doubling formulas for symmetric bilinear torsions and extend the result in \cite{SZ} to $\mathbb{Z}_{2}$-
equivariant case. In Section 5, we
prove the Cheeger-M\"{u}ller type theorem in current case. In Section 6, we compare the symmetric bilinear analytic
torsion with the Ray-Singer analytic torsion on manifolds with boundary.

\section{Symmetric bilinear analytic torsion on manifolds with boundary}

In this section, we will review the definition of
the symmetric bilinear analytic torsion on manifolds with boundary and the anomaly
formula of it.

Let $X$ be a compact oriented Riemannian manifold
with boundary $Y$. Let $F$ be a flat complex vector bundle over $X$. We assume that there is a non-degenerate
symmetric bilinear form $b^{F}$ on $F$. If the Riemannian metric $g^{TX}$ and the form $b^{F}$ are of product
structure near the boundary $Y$.
By the $(g^{TX},b^{F})$, we can define a non-degenerate symmetric bilinear form on $\Omega^*(X,F)$, i.e.,
for $\omega_{1},\omega_{2}\in\Omega^*(X,F)$, then
\begin{align}\label{0.1}
\langle \omega_{1},\omega_{2}\rangle_{b}=\int_{X}{\rm Tr}(\omega_{1}\wedge *_{b^{F}}\omega_{2}),
\end{align}
where $*_{b^{F}}=*\otimes b^{F}:\Omega^{*}(X,F)\to \Omega^*(X,F')$ and $b^{F}:F\to F'$ is induced by the form $b^{F}$, $*$ is the Hodge star operator (cf. \cite{Z}).
Let $d^{F_{\#}}_{b}$ be the formal adjoint of $d^{F}$ with respect to the form in (\ref{0.1}). Then we have the Laplacian
operator
\begin{align}
\Delta_{b}^{F}=(d^{F}+d_{b}^{F_{\#}})^{2}=d^{F}d_{b}^{F_{\#}}+d_{b}^{F_{\#}}d^{F}.
\end{align}
We can impose the relative boundary condition or the absolute boundary condition on $Y$ for the operator $\Delta_{b}$
(cf. \cite{M}). That is, for $\omega\in\Omega^*(X,F)$,
$$i^*\omega=0,\ \ i^*d_{b}^{F_{\#}}\omega=0$$
or
$$i^* *_{b^{F}}\omega=0,\ \ i^*d^{F'_{\#}}_{b}*_{b^{F}}\omega=0.$$
We denote by $\Omega^*(X,F)_{r}$(resp. $\Omega^*(X,F)_{a}$) the complex $(\Omega^*(X,F),d^{F})$ with the relative (resp.
absolute) boundary condition. Then we have
$$H^{\bullet}(\Omega^*(X,F)_{r},d^{F})\cong H^{\bullet}(X,Y,F),\ \
H^{\bullet}(\Omega^*(X,F)_{a},d^{F})\cong H^{\bullet}(X,F).$$

For any $k\geq 0$, let $\Omega^*_{[0,k]}(X,F)_{r}$ (resp. $\Omega^*_{[0,k]}(X,F)_{a}$) denote the generalized eigenspace
of $\Delta_{b}$ on $\Omega^*(X,F)_{r}$ (resp. $\Omega^*(X,F)_{a}$) with respect to the generalized eigenvalues
with the absolute value in $[0,k]$. Let $b_{{\rm det}H^{\bullet}(\Omega_{[0,k]}^*(X,F)_{r})}$ (resp.
$b_{{\rm det}H^{\bullet}(\Omega_{[0,k]}^*(X,F)_{a})}$) be the induced symmetric bilinear form on
${\rm det}H^{\bullet}(\Omega_{[0,k]}^*(X,F)_{r})$ (resp. ${\rm det}H^{\bullet}(\Omega_{[0,k]}^*(X,F)_{a})$) via the
canonical isomorphisms
$${\rm det}H^{\bullet}(\Omega_{[0,k]}^*(X,F)_{r})\cong {\rm det}(\Omega_{[0,k]}^*(X,F)_{r}),$$
$${\rm det}H^{\bullet}(\Omega_{[0,k]}^*(X,F)_{a})\cong {\rm det}(\Omega_{[0,k]}^*(X,F)_{a}).$$
For the subcomplex $(\Omega^*_{[k,+\infty)}(X,F)_{r},d^{F})$ (resp. $(\Omega^*_{[k,+\infty)}(X,F)_{a},d^{F})$), we have
the regulized zeta-determinant (cf. \cite{Gb}, \cite{Se67}, \cite{Se69})
$${\rm det}\left(\Delta_{b}^{F}|_{\Omega^i_{[k,+\infty)}(X,F)_{r}}\right)\ {\rm and}\
{\rm det}\left(\Delta_{b}^{F}|_{\Omega^i_{[k,+\infty)}(X,F)_{a}}\right).$$
Then we can define the symmetric bilinear torsion $b_{{\rm det}H^{\bullet}(X,Y,F)}$
(resp. $b_{{\rm det}H^{\bullet}(X,F)}$) on ${\rm det}H^{\bullet}(X,Y,F)$ (resp. ${\rm det}H^{\bullet}(X,F)$)
by
\begin{align}\label{2.3}
b_{{\rm det}H^{\bullet}(X,Y,F)}=b_{{\rm det}H^{\bullet}(\Omega_{[0,k]}^*(X,F)_{r})}\cdot
\prod_{i=0}^{{\rm dim}X}\left({\rm det}\left(\Delta_{b}^{F}|_{\Omega^i_{[k,+\infty)}(X,F)_{r}}\right)\right)^{(-1)^ii}
\end{align}
and
\begin{align}\label{2.4}
b_{{\rm det}H^{\bullet}(X,F)}=b_{{\rm det}H^{\bullet}(\Omega_{[0,k]}^*(X,F)_{a})}\cdot
\prod_{i=0}^{{\rm dim}X}\left({\rm det}\left(\Delta_{b}^{F}|_{\Omega^i_{[k,+\infty)}(X,F)_{a}}\right)\right)^{(-1)^ii},
\end{align}
which are independent of $k\geq 0$.

Let $\theta(F,b^{F})\in\Omega^{1}(M)$ be the Kamber-Tondeur form defined by (cf. \cite[(4)]{BH1})
$$\theta(F,b^{F})={\rm Tr}\left[(b^{F})^{-1}\nabla^{F}b^{F}\right].$$

Now we state the anomaly formulas. If $(g_{0}^{TX},b_{0}^{F})$ and $(g_{1}^{TX},b_{1}^{F})$ are two couple of metric and
symmetric bilinear form which are of product structure near the boundary (cf. (\ref{1.1}), (\ref{1.2})) and in a same homotopy class . Then by
\cite[Theorem 3]{M}, we have
\begin{multline}\label{1.3}
\log\left({{b_{0,{\rm det}H^{\bullet}(X,Y,F)}}\over{b_{1,{\rm det}H^{\bullet}(X,Y,F)}}}\right)=\int_{X}\log
\left({{b_{0,{\rm det}F}}\over{b_{1,{\rm det}F}}}\right)e\left(TX,\nabla_{1}^{TX}\right)\\
-\int_{X}\theta(F,b_{0}^{F})(\nabla f)^*\left(\psi(TX,\nabla_{0}^{TX})-\psi(TX,\nabla_{1}^{TX})\right)\\
-{1\over 2}\int_{Y}\log \left({{b_{0,{\rm det}F}}\over{b_{1,{\rm det}F}}}\right)e\left(TY,\nabla_{1}^{TY}\right)
-{1\over 2}\int_{Y}\widetilde{e}(TY,\nabla^{TY}_{1},\nabla^{TY}_{0})\theta(F,b^{F}_{0}),
\end{multline}
and
\begin{multline}\label{1.4}
\log\left({{b_{0,{\rm det}H^{\bullet}(X,F)}}\over{b_{1,{\rm det}H^{\bullet}(X,F)}}}\right)=\int_{X}\log
\left({{b_{0,{\rm det}F}}\over{b_{1,{\rm det}F}}}\right)e\left(TX,\nabla_{1}^{TX}\right)\\
-\int_{X}\theta(F,b_{0}^{F})(\nabla f)^*\left(\psi(TX,\nabla_{0}^{TX})-\psi(TX,\nabla_{1}^{TX})\right)\\
+{1\over 2}\int_{Y}\log \left({{b_{0,{\rm det}F}}\over{b_{1,{\rm det}F}}}\right)e\left(TY,\nabla_{1}^{TY}\right)
+{1\over 2}\int_{Y}\widetilde{e}(TY,\nabla^{TY}_{1},\nabla^{TY}_{0})\theta(F,b^{F}_{0}),
\end{multline}
where $e(TX,\nabla^{TX})$ is the Euler form, $\widetilde{e}(TY,\nabla^{TY}_{1},\nabla^{TY}_{0})$ is the Chern-Simons class (cf. \cite[Chapter 4]{BZ1}) and $\psi(TX,\nabla^{TX})$ is the Mathai-Quillen current on $TX$ constructed in \cite[Chapter 3]{BZ1}.

\begin{remark}
Throughout this paper, for complex numbers $a,b\in\mathbb{C}$,  $\log a= b$ means that
$a=e^{b}$.
\end{remark}

\begin{remark}
In \cite{M}, $(g^{TX},b^{F})$ does not need the product structure near the boundary.
\end{remark}

\begin{remark}
If ${\rm dim}X=m$ is odd, then the first two terms in the right hand side of (\ref{1.3}) and (\ref{1.4}) vanish.
\end{remark}

\section{Symmetric bilinear Milnor torsion on manifolds with boundary}
Let $f$ be a Morse function on $X$ and $f|_{Y}$ be the restriction of $f$ on $Y$. Set
\begin{align}
B=\{x\in X; df(x)=0\},\ \ B_{\partial}=\{x\in Y; d(f|_{Y})(x)=0\}.
\end{align}
For $x\in B$, let ${\rm ind}(x)$ be the index of $f$ at $x$, i.e., the number of negative eigenvalues of the quadratic
form $d^{2}f(x)$ on $T_{x}X$.

Consider the differential equation
\begin{align}
{{\partial y}\over{\partial t}}=-\nabla f(y),
\end{align}
and denote by $(\psi_{t})$ the associated flow.

For $x\in B$, the unstable cell $W^{u}(x)$ and the stable cell $W^{s}(x)$ of $x$ are defined by
\begin{align}
W^{u}(x)=\{y\in X;\lim_{t\to -\infty}\psi_{t}(y)=x\},
\end{align}
$$W^{s}(x)=\{y\in X;\lim_{t\to +\infty}\psi_{t}(y)=x\}.$$
The Smale transversality condition means that
\begin{align}\label{6}
{\rm for}\ x,y\in B, \ x\neq y,\ W^{u}(x)\ {\rm and}\ W^{s}(y)\ {\rm intersect}\ {\rm transversally}.
\end{align}

Let $\mathfrak{n}$ be the normal bundle to $Y$ in $X$.

\begin{lemma}(\cite[Lemma 1.5]{BM})\label{t1.1}
There exists a Morse function $f$ on $X$ such that $f|_{Y}$ is a Morse function on $Y$, $B_{\partial}=B\cap Y$,
and $d^{2}f(x)|_{\mathfrak{n}}>0$ for $x\in B_{\partial}$. Moreover, there exists a gradient vector field $\nabla f$ of
$f$, verifying the Smale transversality condition (\ref{6}) and $\nabla f|_{Y}\in TY$.
\end{lemma}

From now on, we choose a Morse function $f$ on $X$ fulfilling the condition of Lemma \ref{t1.1}.
For $x\in B_{\partial}$, set
\begin{align}
W^{u}_{Y}(x)=W^{u}(x)\cap Y,\ \ W^{s}_{Y}(x)=W^{s}(x)\cap Y.
\end{align}
As $d^{2}f(x)|_{\mathfrak{n}}>0$, for $x\in B_{\partial}$, and $\nabla f|_{\partial X}\in T\partial X$, we know that
$\nabla f|_{\partial X}$ verifies also the Smale transversality condition (\ref{6}).

For $x\in B$, we denote by $[W^{u}(x)]$ the complex line generated by $W^{u}(x)$, and by $[W^{u}(x)]^*$ the dual line.
Set
\begin{align}
C_{j}(W^{u},F^*)=\bigoplus_{x\in B,{\rm ind}(x)=j}[W^{u}(x)]\otimes F_{x}^*,
\end{align}
$$C_{j}(W^{u}_{Y},F^*)=\bigoplus_{x\in B\cap Y, {\rm ind}(x)=j}[W^{u}(x)]\otimes F_{x}^*.$$
There is a map $\partial:C_{j}(W^{u},F^*)\to C_{j-1}(W^{u},F^*)$ with $\partial^{2}=0$, which defines the Thom-Smale
complex $(C_{\bullet}(W^{u},F^*),\partial)$ (cf. \cite{Sm}, \cite[(1.30)]{BZ1}), which calculates the
homology $H_{\bullet}(X,F^*)$. As $d^{2}f(x)|_{\mathfrak{n}}>0$, for $x\in B_{\partial}$, and $\nabla f|_{\partial X}
\in T\partial X$, we know that
\begin{align}
\partial C_{j}(W^{u}_{Y},F^*)\subset C_{j-1}(W^{u}_{Y},F^*),
\end{align}
thus the Thom-Smale complex $(C_{\bullet}(W^{u}_{Y},F^*),\partial)$ is a sub-complex of $(C_{\bullet}(W^{u},F^*),\partial)$.
Let $(C^{\bullet}(W^{u},F),\widetilde{\partial})$ and $(C^{\bullet}(W^{u}_{Y},F),\widetilde{\partial})$ be the dual complex
of $(C_{\bullet}(W^{u},F^*),\partial)$ and $(C_{\bullet}(W^{u}_{Y},F^*),\partial)$, respectively. Then the complex
$(C^{\bullet}(W^{u},F),\widetilde{\partial})$ and the complex $(C^{\bullet}(W^{u}_{Y},F),\widetilde{\partial})$ calculate the cohomology
$H^{\bullet}(X,F)$ and $H^{\bullet}(Y,F)$.

Let $\jmath$ be the natural morphism of complexes
\begin{align}
\jmath: C^{\bullet}(W^{u},F)\to C^{\bullet}(W^{u}_{Y},F).
\end{align}
Define
\begin{align}
C^{\bullet}(W^{u}/W^{u}_{Y},F):={\rm Ker}\jmath,
\end{align}
and denote by $H^{\bullet}(C^{\bullet}(W^{u}/W^{u}_{Y},F),\widetilde{\partial})$ the cohomology of
$(C^{\bullet}(W^{u}/W^{u}_{Y},F),\widetilde{\partial})$. Then we have a canonical isomorphism
\begin{align}
H^{\bullet}(C^{\bullet}(W^{u}/W^{u}_{Y},F),\widetilde{\partial})\cong H^{\bullet}(X,Y,F).
\end{align}

We now introduce a symmetric bilinear form on each $[W^{u}(x)]^*\otimes F_{x}$ such that for any
$f,f'\in F_{x}$,
\begin{align}
\langle W^{u}(x)^*\otimes f,W^{u}(x)^*\otimes f'\rangle=\langle f,f'\rangle_{b^{F_{x}}}.
\end{align}
Let $b^{C^{\bullet}(W^{u}/W^{u}_{Y},F)}$ be the symmetric bilinear form on $C^{\bullet}(W^{u}/W^{u}_{Y},F)$ induced by the
form on $C^{\bullet}(W^{u},F)$. Finally, let $b_{{\rm det}C^{\bullet}(W^{u}/W^{u}_{Y},F)}$ be the form on the complex
line
\begin{align}
{\rm det}C^{\bullet}(W^{u}/W^{u}_{Y},F)=\bigotimes_{j=0}^{{\rm dim}X}
\left({\rm det}C^{j}(W^{u}/W^{u}_{Y},F)\right)^{(-1)^{j}}.
\end{align}

\begin{definition}
Let $b^{M,\nabla f}_{{\rm det}H^{\bullet}(X,Y,F)}$ be the symmetric bilinear form on ${\rm det}H^{\bullet}(X,Y,F)$
corresponding to $b_{{\rm det}C^{\bullet}(W^{u}/W^{u}_{Y},F)}$ via the canonical isomorphism
$${\rm det}H^{\bullet}(X,Y,F)\cong {\rm det}C^{\bullet}(W^{u}/W^{u}_{Y},F).$$
The form $b^{M,\nabla f}_{{\rm det}H^{\bullet}(X,Y,F)}$ will be called the symmetric bilinear Milnor torsion.
\end{definition}

\begin{remark}
We can also define the symmetric bilinear Milnor torsion $b^{M,\nabla f}_{{\rm det}H^{\bullet}(X,F)}$ on ${\rm det}H^{\bullet}(X,F)$.
\end{remark}

\section{Doubling formulas for symmetric bilinear torsions}

In this section, we will study the doubling formulas for symmetric bilinear torsions and extend the main theorem in \cite{SZ} to the $\mathbb{Z}_{2}$-equivariant case.

\subsection{Doubling formula for symmetric bilinear Ray-Singer torsion}
We assume that the Riemannian metric $g^{TX}$ is product near
the boundary $Y$, i.e., there exists a neighborhood $U_{\varepsilon}$ of $Y$ and an identification
$Y\times [0,\varepsilon)\to U_{\varepsilon}$, such that for $(y,x_{m})\in Y\times  [0,\varepsilon)$,
\begin{align}\label{1.1}
g^{TX}|_{(y,x_{m})}=dx_{m}^{2}\oplus g^{TY}(y).
\end{align}
This condition insures that the manifold $\widetilde{X}:=X\cup_{Y}X$ has a canonical Riemannian metric
$g^{T\widetilde{X}}=g^{TX}\cup_{Y}g^{TX}$. The natural involution on $\widetilde{X}$ will be denoted by $\phi$, it
generates a $\mathbb{Z}_{2}$-action on $\widetilde{X}$. Let $j_{k}:X\to \widetilde{X}$ be the natural inclusion into the $k$-th factor, $k=1,2$, which identifies $X$ with $j_{k}(X)$.

Let $F$ be a complex flat vector bundle over $X$ with flat connection $\nabla^{F}$. Suppose that there exists a
non-degenerate symmetric bilinear form $b^{F}$ on $F$. we trivialize $F$ on $U_{\varepsilon}$ using the parallel
transport along the curve $u\in [0,1)\to (y,u\varepsilon)$ defined by the connection $\nabla^{F}$, then we have
$F|_{U_{\varepsilon}}=\pi^*_{\varepsilon}F|_{Y}$, where $\pi_{\varepsilon}:Y\times [0,\varepsilon)\to Y$ is the obvious
projection on the first factor. We also assume that
\begin{align}\label{1.2}
b^{F}=\pi_{\varepsilon}^*b^{F}|_{Y}\ \ {\rm on}\ U_{\varepsilon}.
\end{align}

Let $\widetilde{F}=F\cup_{Y}F$ be the flat complex vector bundle with non-degenerate symmetric bilinear form $b^{\widetilde{F}}$ on $\widetilde{X}$ induced by $(F,b^{F})$.

For the couple $(\widetilde{X},\widetilde{F})$, with the Riemannian metric $g^{T\widetilde{X}}$ and the non-degenerate symmetric bilinear form $b^{\widetilde{F}}$, we denote by $D_{b}$ the operator  defined as in \cite[(2.20)]{SZ}.
Let $\langle\ ,\rangle_{b}$ be the symmetric bilinear form on $\Omega^*(\widetilde{X},\widetilde{F})$ defined as in (\ref{0.1}).

For any $a\geq 0$, let $\Omega^*_{[0,a]}(\widetilde{X},\widetilde{F})$ be the generalized eigenspace corresponding to the generalized eigenvalue of $D_{b}^{2}$ with absolute value in $[0,a]$. Let $b_{\Omega^*_{[0,a]}(\widetilde{X},\widetilde{F})}$ be the induced symmetric bilinear form on $\Omega^*_{[0,a]}(\widetilde{X},\widetilde{F})$. As it in \cite{SZ}, we know that it is non-degenerate. Let $\Omega^*_{[0,a]}(\widetilde{X},\widetilde{F})^{\pm}$ be the $\pm 1$-eigenspace of $\phi$ on $\Omega^*_{[0,a]}(\widetilde{X},\widetilde{F})$. Let $b_{\Omega^*_{[0,a]}(\widetilde{X},\widetilde{F})^{\pm}}$ be the induced symmetric bilinear forms on $\Omega^*_{[0,a]}(\widetilde{X},\widetilde{F})^{\pm}$ from $b_{\Omega^*_{[0,a]}(\widetilde{X},\widetilde{F})}$. It is easily to see that $\Omega^*_{[0,a]}(\widetilde{X},\widetilde{F})^{+}$ and $\Omega^*_{[0,a]}(\widetilde{X},\widetilde{F})^{-}$ are orthogonal with respect to $b_{\Omega^*_{[0,a]}(\widetilde{X},\widetilde{F})}$, then $b_{\Omega^*_{[0,a]}(\widetilde{X},\widetilde{F})^{\pm}}$ are non-degenerate.

Let $\Omega^*_{(a,+\infty)}(\widetilde{X},\widetilde{F})$ be the $\langle\ ,\rangle_{b}$-orthogonal complement of $\Omega^*_{[0,a]}(\widetilde{X},\widetilde{F})$. For any $0\leq i\leq m$, let $D^{2}_{b,i}$ be the restriction of $D_{b}^{2}$ on $\Omega^{i}(\widetilde{X},\widetilde{F})$. Then for $g\in\mathbb{Z}_{2}$, we can define the regularized
equivariant zeta determinant
\begin{align}
{\rm det}_{g}'\left(D^{2}_{b,(a,+\infty),i}\right)=\exp\left(-\left.{{\partial}\over{\partial s}}\right|_{s=0}{\rm Tr}\left[g\left(D^{2}_{b,i}|_{\Omega^*_{(a,+\infty)}(\widetilde{X},\widetilde{F})}\right)^{-s}\right]\right).
\end{align}

Let $b_{{\rm det}(H^{\bullet}(\widetilde{X},\widetilde{F})^{\pm})}$ be the symmetric bilinear form on ${\rm det}(H^{\bullet}(\widetilde{X},\widetilde{F})^{\pm})$ induced by the finite dimensional subcomplex of the de Rham complex according the $\pm 1$-eigenvalue. Then for $\mu=(\mu_{1},\mu_{2})\in{\rm det}(H^{\bullet}(\widetilde{X},\widetilde{F}),\mathbb{Z}_{2})$, $g\in\mathbb{Z}_{2}$, the equivariant symmetric bilinear Ray-Singer torsion is defined by
\begin{multline}
b^{\rm RS}_{{\rm det}(H^{\bullet}(\widetilde{X},\widetilde{F}),\mathbb{Z}_{2})}(\mu)(g)
=b_{{\rm det}(H^{\bullet}(\Omega^*_{[0,a]}(\widetilde{X},\widetilde{F}))^{+})}(\mu_{1})\\
\cdot b^{\chi(g)}_{{\rm det}(H^{\bullet}(\Omega^*_{[0,a]}(\widetilde{X},\widetilde{F}))^{-})}(\mu_{2})\cdot
\prod_{i=1}^{{\rm dim}X}\left({\rm det}_{g}'\left(\widetilde{D}^{2}_{b,(a,+\infty),i}\right)\right)^{(-1)^ii},
\end{multline}
which is independent of the choice of $a\geq 0$ (cf. \cite[Theorem 1.4]{BZ2} and \cite[Proposition 4.7]{BH1}).

Denote by $(b_{(a,+\infty),(\widetilde{X},g^{T\widetilde{X}},b^{\widetilde{F}})})(g)$ the product
$$\prod_{i=1}^{{\rm dim}X}\left({\rm det}_{g}'\left(\widetilde{D}^{2}_{b,(a,+\infty),i}\right)\right)^{(-1)^ii}$$
and similarly by $b_{(a,+\infty),(X,Y,g^{TX},b^{F})}$, $b_{((a,+\infty)),(X,g^{TX},b^{F})}$ the corresponding parts in
(\ref{2.3}) and (\ref{2.4}). Let $\mathbb{C}^{+}$, $\mathbb{C}^{-}$ be the trivial and nontrivial one dimension complex $\mathbb{Z}_{2}$-representation, respectively. Then by the same argument in \cite[Proposition 2.1]{BM}, we have

\begin{proposition}\label{t3.3}(Doubling formula for symmetric bilinear analytic torsion)
For $\lambda\geq 0$, let $\Omega^*_{\{\lambda\}}(\widetilde{X},\widetilde{F})$ be the generalized eigenspace of $D_{b}^{2}$ corresponding to the generalized eigenvalue with absolute value equals $\lambda$,  we have a $\mathbb{Z}_{2}$-equivariant map
\begin{align}\label{3.10}
\widetilde{\phi}:\Omega_{\{\lambda\}}^{p}(\widetilde{X},\widetilde{F})\to \Omega^{p}_{\{\lambda\}}(X,F)\otimes \mathbb{C}^{+}
\oplus \Omega_{\{\lambda\}}^{p}(X,Y,F)\otimes \mathbb{C}^{-},
\end{align}
$$\widetilde{\phi}(\sigma)={{\sqrt{2}}\over 2}(\sigma+\phi^*\sigma)|_{X}\otimes 1_{\mathbb{C}^{+}}+{\sqrt{2}\over 2}
(\sigma-\phi^*\sigma)|_{X}\otimes 1_{\mathbb{C}^{-}}.$$
The map $\widetilde{\phi}$ preserves the non-degenerate symmetric bilinear forms. In particular, with $\chi$ the nontrivial
character of $\mathbb{Z}_{2}$, we have for $g\in\mathbb{Z}_{2}$,
\begin{align}\label{3.191}
\left(b_{(0,+\infty),(\widetilde{X},g^{T\widetilde{X}},b^{\widetilde{F}})}\right)(g)=b_{(0,+\infty),(X,g^{TX},b^{F})}
\cdot \left(b_{(0,+\infty),(X,Y,g^{TX},b^{F})}\right)^{\chi(g)}.
\end{align}
\end{proposition}

\begin{proof}
By the same argument in \cite[Proposition 1.27]{Luc}, one finds that $\widetilde{\phi}$ is well-defined and injective.
For the surjectivity, we need to show that for $\omega\in\Omega^{p}_{\{\lambda\}}({X},{F})$, $\widetilde{\omega}=\omega$ on $X$, and $\widetilde{\omega}=\phi^*\omega$ on $\phi(X)$ is a smooth form on $\widetilde{X}$ with coefficients in $\widetilde{F}$, and thus $\widetilde{\omega}\in\Omega^{p}_{\{\lambda\}}(\widetilde{X},\widetilde{F})$. First as the case in \cite[Proposition 1.27]{Luc}, $\widetilde{\omega}$ is smooth on $\widetilde{X}\backslash Y$ and continuous on $\widetilde{X}$. Let $P_{\{\lambda\}}$ be the spectral projection onto $\Omega^{*}_{\{\lambda\}}(\widetilde{X},\widetilde{F})$, then we have the decomposition
\begin{align}
L^{2}\left(\Omega^*(\widetilde{X},\widetilde{F})\right)=
P_{\{\lambda\}}\left(L^{2}\left(\Omega^*(\widetilde{X},\widetilde{F})\right)\right)\oplus (Id-P_{\{\lambda\}})\left(L^{2}\left(\Omega^*(\widetilde{X},\widetilde{F})\right)\right).
\end{align}
So that
\begin{align}
\widetilde{\omega}=P_{\{\lambda\}}(\widetilde{\omega})+(Id-P_{\{\lambda\}})(\widetilde{\omega}).
\end{align}

Let $i:X\to \widetilde{X}=X\cup_{Y}X$ be the inclusion onto the first summand.
Then similar as it in \cite[Proposition 1.27]{Luc}, for any $\eta\in\Omega^*(\widetilde{X},\widetilde{F})$, we have
\begin{align}\label{3.21}
\langle \widetilde{\omega},\eta\rangle_{\widetilde{b}}=\langle \omega,i^*\eta\rangle_{b}+\langle \omega,i^*\phi^*\eta\rangle_{b}.
\end{align}
Then by (\ref{3.21}), the similar discussion in \cite[Proposition 1.27]{Luc} and the fact that for $\lambda\neq \mu$, $\Omega_{\{\lambda\}}^*(X,F)_{a}$ and $\Omega_{\{\mu\}}^*(X,F)_{a}$ are $\langle\ ,\rangle_{b}$-orthogonal (cf. \cite{M}), one finds
\begin{align}
\langle(Id-P_{\{\lambda\}})(\widetilde{\omega}),\eta\rangle_{\widetilde{b}}=
\langle\widetilde{\omega},(Id-P_{\{\lambda\}})\eta\rangle_{\widetilde{b}}=0,
\end{align}
then $(Id-P_{\{\lambda\}})(\widetilde{\omega})=0$. So we have
$\widetilde{\omega}=P_{\{\lambda\}}(\widetilde{\omega})$. Since $\widetilde{\omega}$ is continuous, then we get that $\widetilde{\omega}$ is smooth.
\end{proof}

By Proposition \ref{t3.3} for $\lambda=0$, we have a natural isomorphism of $\mathbb{Z}_{2}$-vector spaces

$$\widetilde{\phi}:\Omega^{\bullet}_{\{0\}}(\widetilde{X},\widetilde{F})\to \Omega^{\bullet}_{\{0\}}(X,F)_{a}\otimes \mathbb{C}^{+}\oplus \Omega^{\bullet}_{\{0\}}(X,F)_{r}\otimes \mathbb{C}^{-},$$
\begin{align}\label{5.31}
\widetilde{\phi}(\sigma)={\sqrt{2}\over 2}\cdot (\sigma+\phi^*\sigma)|_{X}+{\sqrt{2}\over 2}\cdot (\sigma-\phi^*\sigma)|_{X}.
\end{align}
Then $\widetilde{\phi}$ preserves the non-degenerate symmetric bilinear form. From \cite[Proposition 3]{M}, we have
that the inclusion $i:\Omega_{\{0\}}^*(X,F)_{r}\to \Omega^*(X,F)_{r}$ induces an isomorphism on cohomology. We also denote by $\widetilde{\phi}$ the induced map on the cohomology.

Let $b_{{\rm det}H^{\bullet}(\Omega^{\bullet}_{\{0\}}(\widetilde{X},\widetilde{F})^{\pm})}$ be the induced non-degenerate symmetric bilinear form on ${\rm det}H^{\bullet}(\Omega^{\bullet}_{\{0\}}(\widetilde{X},\widetilde{F})^{\pm})$. Then for $\mu=(\mu_{1},\mu_{2})\in {\rm det}(H^{\bullet}{(\widetilde{X},\widetilde{F}}),\mathbb{Z}_{2})$, $g\in\mathbb{Z}_{2}$, the equivariant Ray-Singer symmetric bilinear torsion is defined by
\begin{multline}\label{5.71}
\log \left(b^{\rm RS}_{{\rm det}(H^{\bullet}(\widetilde{X},\widetilde{F}),\mathbb{Z}_{2})}(\mu)\right)(g)=\log\left( b_{{\rm det}(H^{\bullet}(\Omega^{\bullet}_{\{0\}}(\widetilde{X},\widetilde{F}))^{+})}(\mu_{1})\right)\\+\chi(g)\log \left(b_{{\rm det}(H^{\bullet}(\Omega^{\bullet}_{\{0\}}(\widetilde{X},\widetilde{F}))^{-})}(\mu_{2})\right)+\log \left(b_{(0,+\infty),(\widetilde{X},g^{T\widetilde{X}},b^{\widetilde{F}})}\right)(g).
\end{multline}

Note that $\widetilde{\phi}$ in (\ref{5.31}) induces isomorphisms
$$\widetilde{\phi}_{1}:H^{\bullet}(\widetilde{X},\widetilde{F})^{+}\to H^{\bullet}(X,F),\
\widetilde{\phi}_{2}:H^{\bullet}(\widetilde{X},\widetilde{F})^{-}\to H^{\bullet}(X,F).$$

By (\ref{3.191}), (\ref{5.31}) and (\ref{5.71}), for $\mu=(\mu_{1},\mu_{2})\in{\rm det}(H^{\bullet}(\widetilde{X},\widetilde{F}),\mathbb{Z}_{2})$, $g\in\mathbb{Z}_{2}$,
\begin{multline}\label{5.82}
\log \left(b^{\rm RS}_{{\rm det}(H^{\bullet}(\widetilde{X},\widetilde{F}),\mathbb{Z}_{2})}(\mu)\right)(g)\\=\log \left(b^{\rm RS}_{{\rm det}H^{\bullet}(X,F)}(\widetilde{\phi}_{1}\mu_{1})\right)+\chi(g)\log \left(b^{\rm RS}_{{\rm det}H^{\bullet}(X,Y,F)}(\widetilde{\phi}_{2}\mu_{2})\right).
\end{multline}

From (\ref{5.82}) and the anomaly formula \cite[Theorem 3]{M}, one can get the following anomaly formula
\begin{theorem}\label{3.4}
Let $g^{TX}_{u}$ be a smooth one-parameter Riemannian metrics on $X$ which is product near $Y$ and $b_{u}^{F}$ is a smooth one-parameter non-degenerate symmetric bilinear forms which is product near $Y$ such that $g_{0}^{TX}=g^{TX}$, $g_{1}^{TX}=g'^{TX}$ and $b_{0}^{F}=b^{F}$, $b_{1}^{F}=b'^{F}$, then The following identity holds,
\begin{multline}\label{3.192}
\left({{b'^{\rm RS}_{{\rm det}(H^{\bullet}(\widetilde{X},\widetilde{F}),\mathbb{Z}_{2})}}\over{
b^{\rm RS}_{{\rm det}(H^{\bullet}(\widetilde{X},\widetilde{F}),\mathbb{Z}_{2})}}}\right)(\phi)=\exp\left(\int_{Y}\log\left({\rm det}\left(\left(b^{F}\right)^{-1}b'^{F}\right)\right)e\left(TY,\nabla^{TY}\right)\right)\\
\cdot \exp\left(-\int_{Y}\theta\left(F,b'^{F}\right)\widetilde{e}\left(TY,\nabla^{TY},\nabla'^{TY}\right)\right).
\end{multline}
\end{theorem}

\subsection{Doubling formula for symmetric bilinear Milnor torsion}

Let $\widetilde{F}=F\cup_{Y}F$ be the flat complex vector bundle with non-degenerate symmetric bilinear form $b^{\widetilde{F}}$ on $\widetilde{X}=X\cup_{Y}X$ induced by $(F,b^{F})$.

Let $f$ be a Morse function on $X$ satisfying Lemma \ref{t1.1}, which is induced by a $\mathbb{Z}_{2}$-equivariant Morse function $\widetilde{f}$ on $\widetilde{X}$ with critical set $\widetilde{B}=\{x\in\widetilde{X};d\widetilde{f}(x)=0\}$. Let $\widetilde{W}^{u}(x)$ be the unstable set of $x\in\widetilde{B}\subset \widetilde{X}$.

Let $C^{\bullet}(\widetilde{W}^{u},\widetilde{F})^{\pm}$ and $H^{\bullet}(\widetilde{X},\widetilde{F})^{\pm}$ be the $\pm 1$-eigenspaces of the $\mathbb{Z}_{2}$-action induced by $\phi$ on $C^{\bullet}(\widetilde{W}^{u},\widetilde{F})$ and $H^{\bullet}(\widetilde{X},\widetilde{F})$; then $H^{\bullet}(\widetilde{X},\widetilde{F})^{\pm}$ is the cohomology of the complex $(C^{\bullet}(\widetilde{W}^{u},\widetilde{F})^{\pm},\widetilde{\partial})$.

Let $\mathbb{C}^{+}$, $\mathbb{C}^{-}$ be the trivial and nontrivial one dimension complex $\mathbb{Z}_{2}$-representation, respectively, and let $1_{\mathbb{C}^{+}}$, $1_{\mathbb{C}^{-}}$ be their unit elements.

Following \cite[(1.10)]{BZ2}, we define
\begin{align}
{\rm det}\left(H^{\bullet}(\widetilde{X},\widetilde{F}),\mathbb{Z}_{2}\right)={\rm det}\left(H^{\bullet}(\widetilde{X},\widetilde{F})^{+}\right)\otimes \mathbb{C}^{+}\oplus {\rm det}\left(H^{\bullet}(\widetilde{X},\widetilde{F})^{-}\right)\otimes \mathbb{C}^{-}.
\end{align}

Let $b_{{\rm det}H^{\bullet}(\widetilde{X},\widetilde{F})^{\pm}}$ be the symmetric bilinear form on ${\rm det}H^{\bullet}(\widetilde{X},\widetilde{F})^{\pm}$ defined via the canonical isomorphism
$${\rm det}H^{\bullet}(\widetilde{X},\widetilde{F})^{\pm}\cong {\rm det}C^{\bullet}(\widetilde{W}^{u},\widetilde{F})^{\pm}.$$
For $\mu=(\mu_{1},\mu_{2})\in{\rm det}(H^{\bullet}(\widetilde{X},\widetilde{F}),\mathbb{Z}_{2})$, $\phi\in\mathbb{Z}_{2}$, and $\chi$ the nontrivial character of $\mathbb{Z}_{2}$, we introduce the equivariant symmetric bilinear Milnor torsion by
\begin{multline}
b^{M,\nabla f}_{{\rm det}(H^{\bullet}(\widetilde{X},\widetilde{F}),\mathbb{Z}_{2})}(\mu)(\phi)=b_{{\rm det}H^{\bullet}(\widetilde{X},\widetilde{F})^+}(\mu_{1})\cdot b^{\chi(\phi)}_{{\rm det}H^{\bullet}(\widetilde{X},\widetilde{F})^-}(\mu_{2})\\
=b_{{\rm det}H^{\bullet}(\widetilde{X},\widetilde{F})^+}(\mu_{1})\cdot b^{-1}_{{\rm det}H^{\bullet}(\widetilde{X},\widetilde{F})^-}(\mu_{2}).
\end{multline}

We have a $\mathbb{Z}_{2}$-equivariant isomorphism of complexes
\begin{align}\label{3.33}
\gamma: C^{\bullet}(W^{u},F)\otimes \mathbb{C}^{+}\oplus C^{\bullet}(W^{u}/W^{u}_{Y},F)\otimes
\mathbb{C}^{-}\to C^{\bullet}(\widetilde{W}^{u},\widetilde{F}),
\end{align}
given by
\begin{multline}
\gamma(\mathfrak{a}^*\otimes 1_{\mathbb{C}^{+}}\oplus \mathfrak{b}^*\otimes 1_{\mathbb{C}^{-}})\\=
{\sqrt{2}\over 2}((j_{1}^{-1})^*\mathfrak{a}^*+(j_{2}^{-1})^*\mathfrak{a})+{\sqrt{2}\over 2}((j_{1}^{-1})^*\mathfrak{b}^*
-(j_{2}^{-1})^*\mathfrak{b}^*),
\end{multline}
which induces a $\mathbb{Z}_{2}$-isomorphism
\begin{align}\label{3.14}
\gamma: H^{\bullet}(X,F)\otimes \mathbb{C}^{+}\oplus H^{\bullet}(X,Y,F)\otimes \mathbb{C}^{-}\to H^{\bullet}
(\widetilde{X},\widetilde{F}).
\end{align}
Note that as complex vector spaces, we have
\begin{align}
C^{j}(W^{u},F)=\bigoplus_{x\in B,{\rm ind}(x)=j}\left[
W^{u}(x)\right]^*\otimes F_{x},
\end{align}
$$C^{j}(W^{u}/W^{u}_{Y},F)=\bigoplus_{x\in B\setminus Y,{\rm ind}(x)=j}\left[W^{u}(x)\right]^*\otimes F_{x}.$$
Then $\gamma$ as a map from $C^{j}(W^{u}/W^{u}_{Y},F)\otimes \mathbb{C}^{+}\oplus C^{j}(W^{u}/W^{u}_{Y},F)\otimes
\mathbb{C}^{-}$ into $C^{\bullet}(\widetilde{W}^{u},\widetilde{F})$ preserves the symmetric bilinear form and for
$\mathfrak{a}^*\in[W^{u}(x)]^*\otimes F_{x}$ with $x\in Y$, we have
\begin{align}\label{3.16}
\gamma(\mathfrak{a}^*\otimes 1_{\mathbb{C}^{+}})=\sqrt{2}\mathfrak{a}^*.
\end{align}

Let $b_{{\rm det}C^{\bullet}(\widetilde{W}^{u},\widetilde{F})^{\pm}}$ be the symmetric bilinear form
on ${\rm det}H^{\bullet}(\widetilde{X},\widetilde{F})^{\pm}$. For $\mu=(\mu_{1},\mu_{2})\in {\rm det}(H^{\bullet}
(\widetilde{X},\widetilde{F}),
\mathbb{Z}_{2})$, $g\in\mathbb{Z}_{2}$, and $\chi$ the nontrivial character of $\mathbb{Z}_{2}$, we have defined
the equivariant Milnor symmetric bilinear torsion by
\begin{align}\label{3.18}
b^{M,\nabla f}_{{\rm det}(H^{\bullet}(\widetilde{X},\widetilde{F}),\mathbb{Z}_{2})}(\mu)(g)=b_{{\rm det}C^{\bullet}
(\widetilde{W}^{u},\widetilde{F})^{+}}(\mu_{1})\cdot (b_{{\rm det}C^{\bullet}
(\widetilde{W}^{u},\widetilde{F})^{-}}(\mu_{2}))^{\chi(g)}.
\end{align}

Now $\gamma$ in (\ref{3.14}) induces isomorphisms
$$\gamma_{1}^{-1}:H^{\bullet}(\widetilde{X},\widetilde{F})^{+}\to H^{\bullet}(X,F),$$
\begin{align}\label{3.19}
\gamma_{2}^{-1}:H^{\bullet}(\widetilde{X},\widetilde{F})^{-}\to H^{\bullet}(X,Y,F).
\end{align}
From (\ref{3.16}), (\ref{3.18}) and (\ref{3.19}), we get for $\mu=(\mu_{1},\mu_{2})\in{\rm det}(H^{\bullet}(\widetilde{X}
,\widetilde{F}),\mathbb{Z}_{2})$, $g\in\mathbb{Z}_{2}$,
\begin{multline}\label{3.28}
\log\left(b^{M,\nabla f}_{{\rm det}(H^{\bullet}(\widetilde{X},\widetilde{F}),\mathbb{Z}_{2})}(\mu)\right)(g)
=\log(2)\sum_{x\in B\cap Y}(-1)^{{\rm ind}(x)}{\rm rk}(F)\\
+\log \left(b^{M,\nabla f}_{{\rm det}H^{\bullet}(X,F)}(\gamma_{1}^{-1}\mu_{1})\right)+\chi(g)\log
\left(b^{M,\nabla f}_{{\rm det}H^{\bullet}(X,Y,F)}(\gamma_{2}^{-1}\mu_{2})\right).
\end{multline}

From (\ref{3.28}) and the anomaly formulas for $b^{M,\nabla f}_{{\rm det}H^{\bullet}(X,F)}$, $b^{M,\nabla f}_{{\rm det}H^{\bullet}(X,Y,F)}$, we easily get the following anomaly formula

\begin{proposition}\label{3.5}
Let $b_{u}^{F}$, $0\leq u\leq 1$, be a smooth one-parameter non-degenerate symmetric bilinear forms which are product near the boundary $Y$, then we have
\begin{align}
\left({{b^{M,\nabla f}_{1,{\rm det}(H^{\bullet}(\widetilde{X},\widetilde{F}),\mathbb{Z}_{2})}}\over{b^{M,\nabla f}_{0,{\rm det}(H^{\bullet}(\widetilde{X},\widetilde{F}),\mathbb{Z}_{2})}}}\right)(\phi)=\prod_{x\in B_{Y}}\left({\rm det}\left({{b_{1}^{F_{x}}}\over{b_{0}^{F_{x}}}}\right)\right)^{(-1)^{{\rm ind}_{Y}(x)}}.
\end{align}
\end{proposition}

\subsection{Comparison of $b^{M,\nabla f}_{{\rm det}(H^{\bullet}(\widetilde{X},\widetilde{F}),\mathbb{Z}_{2})}$
and $b^{\rm RS}_{{\rm det}(H^{\bullet}(\widetilde{X},\widetilde{F}),\mathbb{Z}_{2})}$}

In this section, we will compare $b^{M,\nabla f}_{{\rm det}(H^{\bullet}(\widetilde{X},\widetilde{F}),\mathbb{Z}_{2})}$
and $b^{\rm RS}_{{\rm det}(H^{\bullet}(\widetilde{X},\widetilde{F}),\mathbb{Z}_{2})}$, we will use the method in \cite{SZ}.

\begin{theorem}({\rm Compare\ with\ \cite[Theorem 5.1]{BZ2}})\label{t3.8}
Let $\phi\in\mathbb{Z}_{2}$ be element induced by the involution on $\widetilde{X}$, then the following identity holds,
\begin{multline}
{{b^{\rm RS}_{{\rm det}(H^{\bullet}(\widetilde{X},\widetilde{F}),\mathbb{Z}_{2})}}\over{b^{M,\nabla f}_{{\rm det}(H^{\bullet}(\widetilde{X},\widetilde{F}),\mathbb{Z}_{2})}}}(\phi)=\\
\exp\left(-\int_{Y}\theta(\widetilde{F},b^{\widetilde{F}})(\nabla f)^*\psi(TY,\nabla^ {TY})+{\rm rk}(F)\chi(Y)\log (2)\right).
\end{multline}
\end{theorem}

\begin{remark}
By Theorem \ref{3.4}, Proposition \ref{3.5} and \cite[Section 4]{BZ2}, we can assume that $g^{T\widetilde{X}}$ and $\widetilde{f}$ verifies the condition in \cite[Section 5 b)]{BZ2}. Moreover, we may assume that $b^{\widetilde{F}}$ is flat on an open neighborhood of the zero set $\widetilde{B}$. By \cite[Theorem 5.9]{BH1}, from $b^{\widetilde{F}}$ we can get a Hermitian metric $g^{\widetilde{F}}$, which can also be assumed to be flat on an open neighborhood of the zero set $\widetilde{B}$.
\end{remark}

For any $T\in\mathbb{R}$, let $b_{T}^{\widetilde{F}}$ be the deformed symmetric bilinear form on $\widetilde{F}$
defined by
\begin{align}
b_{T}^{\widetilde{F}}(u,v)=e^{-2Tf}b^{\widetilde{F}}(u,v).
\end{align}
Let $d^{\widetilde{F}_{\#}}_{b_{T}}$ be the associated formal adjoint in the sense of (\ref{0.1}). Set
\begin{align}
D_{b_{T}}=d^{\widetilde{F}}+d^{\widetilde{F}_{\#}}_{b_{T}},\
D_{b_{T}}^{2}=\left(d^{\widetilde{F}}+d^{\widetilde{F}_{\#}}_{b_{T}}\right)^{2}=d^{\widetilde{F}}
d^{\widetilde{F}_{\#}}_{b_{T}}+d^{\widetilde{F}_{\#}}_{b_{T}}d^{\widetilde{F}}.
\end{align}

Let $\Omega^*_{[0,1],T}(\widetilde{X},\widetilde{F})$ be defined as in Section 2 with respect to $D^{2}_{b_{T}}$, and let
$\Omega^*_{[0,1],T}(\widetilde{X},\widetilde{F})^{\bot}$ be the orthogonal complement with respect to the symmetric
bilinear form. Let $P_{T}^{(1,+\infty)}$ be spectral projection onto $\Omega^*_{[0,1],T}(\widetilde{X},\widetilde{F})^{\bot}$.

In the current case, comparing with the notations in \cite[p. 168]{BZ2}, we introduce
$$\chi_{\phi}(\widetilde{F})={\rm rk}(\widetilde{F})\chi(Y),$$
$$\widetilde{\chi}'_{\phi}(\widetilde{F})={\rm rk}(\widetilde{F})\sum_{x\in B_{\partial}}(-1)^{{\rm ind}_{Y}(x)}{\rm ind}(x)={\rm rk}(\widetilde{F})\sum_{x\in B_{\partial}}(-1)^{{\rm ind}_{Y}(x)}{\rm ind}_{Y}(x)$$
and
$${\rm Tr}_{s}^{B_{\partial}}[f]={\rm rk}(\widetilde{F})\sum_{x\in B_{\partial}}(-1)^{{\rm ind}_{Y}(x)}f(x),$$
where we used the fact that $d^{2}f(x)|_{\mathfrak{n}}>0$ for $x\in B_{Y}$.

Let $N$ be the number operator on $\Omega^*(\widetilde{X},\widetilde{F})$ acting on $\Omega^i(\widetilde{X},\widetilde{F})$
by multiplication by $i$. Let $P_{T}^{[0,1]}$ be the restriction of the de Rham map $P_{\infty}$ (cf. \cite[(3.1)-(3.6)]{SZ}) on
$\Omega^*_{[0,1],T}(\widetilde{X},\widetilde{F})$, and let $P_{T}^{[0,1],{\rm det}H}$ be the induced isomorphism on
cohomology.

We now state several intermediate results whose proofs will be given later. Note that $\phi$ acts as ${\rm Id}$ on $\widetilde{F}_{x}$ for
$x\in Y$.

\begin{theorem}\label{t3.4}
The following identity holds:
\begin{multline}\label{3.38}
\lim_{T\to +\infty}{{P_{T}^{[0,1],{\rm det}H}(b_{{\rm det}((H^{\bullet}(\Omega^*_{[0,1],T})),
\mathbb{Z}_{2})})}\over{b^{M,\nabla f}_{{\rm det}H^{\bullet}((\widetilde{X},\widetilde{F}),
\mathbb{Z}_{2})}}}(\phi)\left({T\over \pi}\right)^{{m\over 2}\chi_{\phi}(\widetilde{F})-\widetilde{\chi}'_{\phi}(\widetilde{F})}
\exp(2{\rm Tr}_{s}^{B_{\partial}}[f]T)\\=1.
\end{multline}
\end{theorem}

\begin{theorem}\label{t3.5}
For any $t>0$,
\begin{align}
\lim_{T\to +\infty}{\rm Tr}_{s}\left[\phi N\exp(-tD^{2}_{b_{T}})P_{T}^{(1,+\infty)}\right]=0.
\end{align}
Moreover, for any $d>0$ there exist $c>0$, $C>0$ and $T_{0}\geq 1$ such that for any $t\geq d$ and $T\geq T_{0}$,
\begin{align}\label{3.401}
\left|{\rm Tr}_{s}\left[\phi N\exp\left(-tD^{2}_{b_{T}}\right)P^{(1,+\infty)}_{T}\right]\right|\leq c\exp(-Ct).
\end{align}
\end{theorem}

\begin{theorem}\label{t3.6}
The following identity holds
\begin{align}\label{3.34}
\lim_{T\to +\infty}{\rm Tr}_{s}\left[\phi N P_{T}^{[0,1]}\right]=\widetilde{\chi}'_{\phi}(\widetilde{F}).
\end{align}
Also
\begin{align}\label{3.35}
\lim_{T\to +\infty}{\rm Tr}_{s}\left[\phi D^{2}_{b_{T}}P_{T}^{[0,1]}\right]=0.
\end{align}
\end{theorem}

\begin{theorem}\label{t3.7}
As $t\to 0$, the following identity holds,
\begin{multline}
{\rm Tr}_{s}\left[\phi N\exp(-tD^{2}_{T})\right]={m\over 2}\chi_{\phi}(\widetilde{F})+O(t)\ {\rm if}\ m\ {\rm is}\ {\rm even},\\
={\rm rk}(F)\int_{Y}\int^{B}L \exp\left(-{{\dot{R}^{TY}}\over 2}\right)\cdot {1\over{\sqrt{t}}}+O(\sqrt{t})\ {\rm if}\ m\ {\rm is}\ {\rm odd},
\end{multline}
where $L$ is defined similar as \cite[(3.52)]{BZ1} on $Y$.
\end{theorem}

\begin{theorem}\label{t3.9}
There exist $0\leq \alpha\leq 1$, $C>0$ such that for any $0\leq t\leq \alpha$, $0\leq T\leq {1\over t}$, then
\begin{multline}
\left|{\rm Tr}_{s}\left[\phi N\exp\left(-(tD_{b}+T\widehat{c}(\nabla f))^{2}\right)\right]-{1\over t}
\int_{Y}\int^{B}L \exp(-B_{T^{2}}){\rm rk}(\widetilde{F})\right.\\
\left.-{T\over 2}\int_{Y}\theta_{\phi}(\widetilde{F},b^{\widetilde{F}})
\int^{B}\widehat{df}\exp(-B_{T^{2}})-{m\over 2}\chi_{\phi}(\widetilde{F})\right|\leq Ct.
\end{multline}
\end{theorem}

\begin{theorem}\label{t3.10}
For any $T>0$, the following identity holds
\begin{multline}\label{3.520}
\lim_{t\to 0}\left[\phi N\exp\left(-\left(tD_{b}+{T\over t}\widehat{c}(\nabla f)\right)^{2}\right)\right]\\
={\rm rk}(\widetilde{F})
\cdot \left({1\over{1-e^{-2T}}}\left((1+e^{-2T})\sum_{x\in Y\cap B}(-1)^{{\rm ind}_{Y}(x)}{\rm ind}_{Y}(x)-{\rm dim}Ye^{-2T}
\chi(Y)\right)\right)\\
-{\rm rk}(\widetilde{F}){1\over 2}\left({{\sinh(2T)}\over{\cosh(2T)+1}}-1\right)\chi(Y).
\end{multline}
\end{theorem}

\begin{theorem}\label{t3.11}
There exist $\alpha\in (0,1]$, $c>0$, $C>0$ such that for any $t\in (0,\alpha]$, $T\geq 1$, then
\begin{align}\label{3.532}
\left|{\rm Tr}_{s}\left[\phi N\exp\left(-\left(tD_{b}+{T\over t}\widehat{c}(\nabla f)\right)^{2}\right)\right]
-\widetilde{\chi}'_{\phi}(\widetilde{F})\right|\leq c\exp(-CT).
\end{align}
\end{theorem}

Now we give a proof of Theorem \ref{t3.8}.

\begin{theorem}
The following identity holds,
\begin{multline}\label{3.40}
{{b^{\rm RS}_{{\rm det}(H^{\bullet}(\widetilde{X},\widetilde{F}),\mathbb{Z}_{2})}}\over{b^{M,\nabla f}_{{\rm det}(H^{\bullet}(\widetilde{X},\widetilde{F}),\mathbb{Z}_{2})}}}(\phi)=\\
\exp\left(-\int_{Y}\theta(\widetilde{F},b^{\widetilde{F}})(\nabla f)^*\psi(TY,\nabla^ {TY})+{\rm rk}(F)\chi(Y)\log (2)\right).
\end{multline}
\end{theorem}
\begin{proof}

First of all, by the anomaly formula (\ref{3.192}), for any $T\geq 0$, one has
\begin{multline}\label{3.481}
{{P_{T}^{[0,1],{\rm det}H}\left(b^{\rm RS}_{{\rm det}H^*\left(\Omega^*_{[0,1],T}(\widetilde{X},\widetilde{F}),\mathbb{Z}_{2}\right)}\right)}\over{b^{M,\nabla f}_{{\rm det}(H^*(\widetilde{X},\widetilde{F}),\mathbb{Z}_{2})}}}(\phi)\\
\cdot \prod_{i=0}^{m}\left({\rm det}\left(D^{2}_{b_{T}}|_{\Omega^*_{[0,1],T}(\widetilde{X},\widetilde{F})^{\bot}\cap \Omega^{i}(\widetilde{X},\widetilde{F})}\right)(\phi)\right)^{(-1)^i i}\\
={{P^{{\rm det}H}_{\infty}\left(b^{\rm RS}_{{\rm det}(H^*(\widetilde{X},\widetilde{F}),\mathbb{Z}_{2})}\right)}\over{b^{M,\nabla f}_{{\rm det}(H^*(\widetilde{X},\widetilde{F}),\mathbb{Z}_{2})}}}(\phi)\exp\left(-2T{\rm rk}(F)\int_{Y}fe\left(TY,\nabla^{TY}\right)\right).
\end{multline}

From now on, we will write $a\simeq b$ for $a$, $b\in\mathbb{C}$ if $e^{a}=e^{b}$. Thus, we can rewrite (\ref{3.481}) as
\begin{multline}\label{3.491}
\log\left({{P^{{\rm det}H}_{\infty}\left(b^{\rm RS}_{{\rm det}(H^*(\widetilde{X},\widetilde{F}),\mathbb{Z}_{2})}\right)}\over{b^{M,\nabla f}_{{\rm det}(H^*(\widetilde{X},\widetilde{F}),\mathbb{Z}_{2})}}}(\phi)\right)\\
\simeq \log\left({{P_{T}^{[0,1],{\rm det}H}\left(b^{\rm RS}_{{\rm det}H^*\left(\Omega^*_{[0,1],T}(\widetilde{X},\widetilde{F}),\mathbb{Z}_{2}\right)}\right)}\over{b^{M,\nabla f}_{{\rm det}(H^*(\widetilde{X},\widetilde{F}),\mathbb{Z}_{2})}}}(\phi)\right)\\
+\sum_{i=0}^{m}(-1)^i i\log \left({\rm det}\left(D^{2}_{b_{T}}|_{\Omega^*_{[0,1],T}(\widetilde{X},\widetilde{F})^{\bot}\cap \Omega^{i}(\widetilde{X},\widetilde{F})}\right)(\phi)\right)\\
+2T{\rm rk}(F)\int_{Y}fe\left(TY,\nabla^{TY}\right).
\end{multline}

Let $T_{0}>0$ be as in Theorem \ref{t3.5}. For any $T\geq T_{0}$ and $s\in\mathbb{C}$ with ${\rm Re}(s)\geq m+1$, set
\begin{align}\label{3.501}
\theta_{\phi,T}(s)={1\over{\Gamma(s)}}\int_{0}^{+\infty}t^{s-1}{\rm Tr}_{s}\left[\phi N\exp\left(-tD^{2}_{b_{T}}\right)P_{T}^{(1,+\infty)}\right]dt.
\end{align}
By (\ref{3.401}), $\theta_{\phi,T}(s)$ is well defined and can be extended to a meromorphic function which is holomorphic at $s=0$. Moreover,
\begin{align}\label{3.511}
\sum_{i=0}^{m}(-1)^{i}i\log \left({\rm det}\left(D^{2}_{b_{T}}|_{\Omega^*_{[0,1],T}(\widetilde{X},\widetilde{F})^{\bot}\cap \Omega^{i}(\widetilde{X},\widetilde{F})}\right)(\phi)\right)\simeq
-\left.{{\partial \theta_{\phi,T}(s)}\over{\partial s}}\right|_{s=0}.
\end{align}

Let $d=\alpha^{2}$ with $\alpha$ being as in Theorem \ref{t3.11}. From (\ref{3.501}) and Theorems \ref{t3.5}-\ref{t3.7}, one finds that
\begin{multline}\label{3.521}
\lim_{T\to +\infty}\left.{{\partial \theta_{\phi,T}(s)}\over{\partial s}}\right|_{s=0}\\
=\lim_{T\to +\infty}\int_{0}^{d}\left({\rm Tr}_{s}\left[\phi N\exp\left(-tD^{2}_{b_{T}}\right)\right]-{{a_{-1}}\over{\sqrt{t}}}-{m\over 2}\chi_{\phi}(\widetilde{F})\right){{dt}\over{t}}\\
-{{2a_{-1}}\over{\sqrt{d}}}-\left(\Gamma'(1)-\log d\right)\left({m\over 2}\chi_{\phi}(\widetilde{F})-\widetilde{\chi}'_{\phi}(\widetilde{F})\right),
\end{multline}
where we denote for simplicity that
$$a_{-1}={\rm rk}(F)\int_{Y}\int^{B}L\exp\left(-{{\dot{R}^{TY}}\over 2}\right).$$

To study the first term in the right hand side of (\ref{3.521}), we observe first that for any $T\geq 0$, one has
\begin{align}\label{3.531}
e^{-T\widetilde{f}}D^{2}_{b_{T}}e^{T\widetilde{f}}=\left(D_{b}+T
\widehat{c}(\nabla\widetilde{f})\right)^{2}.
\end{align}
Thus, one has
\begin{align}\label{3.541}
{\rm Tr}_{s}\left[\phi N\exp\left(-tD^{2}_{b_{T}}\right)^{2}\right]={\rm Tr}_{s}\left[\phi N\exp\left(-t\left(D_{b}+T\widehat{c}(\nabla \widetilde{f})\right)\right)\right].
\end{align}

By (\ref{3.541}), one writes
\begin{multline}\label{3.551}
\int_{0}^{d}\left({\rm Tr}_{s}\left[\phi N\exp\left(-tD^{2}_{b_{T}}\right)\right]-{{a_{-1}}\over{\sqrt{t}}}-{m\over 2}\chi_{\phi}(\widetilde{F})\right){{dt}\over{t}}\\
=2\int_{1}^{\sqrt{dT}}\left({\rm Tr}_{s}\left[\phi N\exp\left(-\left({t\over{\sqrt{T}}}D_{b}+t\sqrt{T}\widehat{c}(\nabla \widetilde{f})\right)^{2}\right)\right]-{\sqrt{T}\over{t}}a_{-1}\right.\\
\left.-{m\over 2}\chi_{\phi}(\widetilde{F})\right){{dt}\over{t}}\\
+2\int_{0}^{{1\over{\sqrt{T}}}}\left({\rm Tr}_{s}\left[\phi N\exp\left(-\left(tD_{b}+tT\widehat{c}(\nabla\widetilde{f})\right)^{2}\right)\right]-{{a_{-1}}
\over{t}}-{m\over 2}\chi_{\phi}(\widetilde{F})\right){{dt}\over{t}}.
\end{multline}

In view of Theorem \ref{t3.9}, we write
\begin{multline}\label{3.561}
\int_{0}^{{1\over{\sqrt{T}}}}\left({\rm Tr}_{s}\left[\phi N\exp\left(-\left(tD_{b}+tT\widehat{c}(\nabla\widetilde{f})\right)^{2}\right)\right]-{{a_{-1}}
\over{t}}-{m\over 2}\chi_{\phi}(\widetilde{F})\right){{dt}\over{t}}\\
=\int_{0}^{1\over{\sqrt{T}}}\left({\rm Tr}_{s}\left[\phi N\exp\left(-\left(tD_{b}+tT\widehat{c}(\nabla\widetilde{f})\right)^{2}\right)\right]\right.\\
\left.-{1\over{t}}{\rm rk}(F)\int_{Y}\int^{B}L\exp\left(-B_{(tT)^{2}}\right)\right.\\
\left.-{{tT}\over 2}\int_{Y}\theta(F,b^{F})\int^{B}\widehat{d\widetilde{f}}\exp\left(-B_{(tT)^{2}}\right)-{m\over 2}\chi_{\phi}(\widetilde{F})\right){{dt}\over t}\\
+\int_{0}^{1\over{\sqrt{T}}}\left({1\over t}{\rm rk}(F)\int_{Y}\int^{B}L\exp\left(-B_{(tT)^{2}}\right)-{{a_{-1}}\over{t}}\right){{dt}\over t}\\
+\int_{0}^{1\over{\sqrt{T}}}{{tT}\over 2}\int_{Y}\theta(F,b^{F})\int^{B}\widehat{d\widetilde{f}}\exp\left(-B_{(tT)^{2}}\right){{dt}\over{t}}.
\end{multline}

By \cite[Definitions 3.6, 3.12 and Theorem 3.18]{BZ1}, one has that, as $T\to +\infty$,
\begin{multline}\label{3.571}
\int_{0}^{1\over{\sqrt{T}}}{{tT}\over 2}\int_{Y}\theta(F,b^{F})\int^{B}\widehat{d\widetilde{f}}\exp\left(-B_{(tT)^{2}}\right){{dt}\over{t}}\to\\
{1\over 2}\int_{Y}\theta(F,b^{F})(\nabla f)^*\psi\left(TY,\nabla^{TY}\right).
\end{multline}

From \cite[(3.54)]{BZ1}, \cite[(3.35)]{SZ} and integration by parts, we have
\begin{multline}\label{3.581}
\int_{0}^{1\over{\sqrt{T}}}\left({1\over t}{\rm rk}(F)\int_{Y}\int^{B}L\exp\left(-B_{(tT)^{2}}\right)-{{a_{-1}}\over{t}}\right){{dt}\over t}\\
=-{\sqrt{T}}{\rm rk}(F)\int_{Y}\int^{B}L\exp(-B_{T})+\sqrt{T}a_{-1}\\
-T{\rm rk}(F)\int_{Y}f\int^{B}\exp(-B_{T})+T{\rm rk}(F)\int_{Y}f\int^{B}\exp(-B_{0}).
\end{multline}

From Theorems \ref{t3.9}, \ref{t3.10}, \cite[Theorem 3.20]{BZ1}, \cite[(7.72) and (7.73)]{BZ1} and the dominate convergence, one finds that as $T\to +\infty$,
\begin{multline}\label{3.591}
\int_{0}^{1\over{\sqrt{T}}}\left({\rm Tr}_{s}\left[\phi N\exp\left(-\left(tD_{b}+tT\widehat{c}(\nabla \widetilde{f})\right)^{2}\right)\right]\right.\\
\left.-{1\over t}{\rm rk}(F)\int_{Y}\int^{B}L\exp\left(-B_{(tT)^{2}}\right)\right.\\
-\left.{{tT}\over 2}\int_{Y}\theta(F,b^{F})\int^{B}\widehat{d\widetilde{f}}\exp\left(-B_{(tT)^{2}}\right)-{m\over 2}\chi_{\phi}(\widetilde{F})\right){{dt}\over t}\\
=\int_{0}^{1}\left({\rm Tr}_{s}\left[\phi N\exp\left(-\left({t\over{\sqrt{T}}}D_{b}+t\sqrt{T}\widehat{c}(\nabla\widetilde{f})
\right)^{2}\right)\right]\right.\\
-{\sqrt{T}\over{t}}{\rm rk}(F)\int_{Y}\int^{B}L\exp\left(-B_{(t\sqrt{T})^{2}}\right)\\
\left.-{{t\sqrt{T}}\over 2}\int_{Y}\theta(F,b^{F})\int^{B}\widehat{d\widetilde{f}}\exp\left(-B_{(t\sqrt{T})^{2}}\right)-{m\over 2}\chi_{\phi}(\widetilde{F})\right){{dt}\over t}\\
\to \int_{0}^{1}\left\{{1\over{1-e^{-2t^{2}}}}\left(\left(1+e^{-2t^{2}}\right)\widetilde{\chi}'_{\phi}(\widetilde{F})-{\rm dim}Ye^{-2t^{2}}\chi_{\phi}(\widetilde{F})\right)\right.\\
-{\rm rk}(\widetilde{F}){1\over 2}\left({{\sinh(2t^{2})}\over{\cosh(2t^{2})+1}}-1\right)\chi(Y)\\
\left.+{1\over{2t^{2}}}{\rm rk}(\widetilde{F})\sum_{x\in B_{Y}}(-1)^{{\rm ind}_{Y}(x)}({\rm dim}Y-2{\rm ind}_{Y}(x))-{m\over 2}\chi_{\phi}(\widetilde{F})\right\}{{dt}\over{t}}\\
={1\over 2}{\rm rk}(\widetilde{F})\left\{\sum_{x\in B_{Y}}(-1)^{{\rm ind}_{Y}(x)}{\rm ind}_{Y}(x)-{1\over 2}\chi(Y){\rm dim}Y\right\}\cdot\int_{0}^{1}\left({{1+e^{-2t}}\over{1-e^{-2t}}}-{1\over t}\right){{dt}\over{t}}\\
-{\rm rk}(\widetilde{F}){1\over 4}\chi(Y)\cdot \int_{0}^{1}{{\sinh(2t)}\over{\cosh(2t)+1}}{{dt}\over{t}}.
\end{multline}

On the other hand, by Theorems \ref{t3.10}, \ref{t3.11} and the dominate convergence, we have that as $T\to +\infty$,
\begin{multline}\label{3.601}
\int_{1}^{\sqrt{Td}}\left({\rm Tr}_{s}\left[\phi N\exp\left(-\left({t\over{\sqrt{T}}}D_{b}+t\sqrt{T}\widehat{c}(\nabla \widetilde{f})\right)^{2}\right)\right]-{{\sqrt{T}}\over t}a_{-1}-{m\over 2}\chi_{\phi}(\widetilde{F})\right){{dt}\over{t}}\\
=\int_{1}^{\sqrt{Td}}\left({\rm Tr}_{s}\left[\phi N\exp\left(-\left({t\over{\sqrt{T}}}D_{b}+t\sqrt{T}\widehat{c}(\nabla \widetilde{f})\right)^{2}\right)\right]-\widetilde{\chi}'_{\phi}(\widetilde{F})\right){{dt}\over{t}}\\
+{1\over 2}\widetilde{\chi}'_{\phi}(\widetilde{F})\log(Td)+a_{-1}\sqrt{T}\left({1\over{\sqrt{Td}}}-1\right)-{m\over 4}\chi_{\phi}(\widetilde{F})\log(Td)\\
=\int_{1}^{+\infty}\left\{{1\over{1-e^{-2t^{2}}}}\left(\left(1+e^{-2t^{2}}\right)\widetilde{\chi}'_{\phi}
(\widetilde{F})-{\rm dim}Ye^{-2t^{2}}\chi_{\phi}(\widetilde{F})\right)\right.\\
\left.-{\rm rk}(\widetilde{F}){1\over 2}\left({{\sinh(2t^{2})}\over{\cosh(2t^{2})+1}}-1\right)\chi(Y)-\widetilde{\chi}'_{\phi}(\widetilde{F})\right\}
{{dt}\over{t}}\\
+{1\over 2}\widetilde{\chi}'_{\phi}(\widetilde{F})\log(Td)+a_{-1}\sqrt{T}\left({1\over{\sqrt{Td}}}-1\right)-{m\over 4}\chi_{\phi}(\widetilde{F})\log(Td)+o(1)\\
={\rm rk}(\widetilde{F})\left\{\sum_{x\in B_{Y}}(-1)^{{\rm ind}_{Y}(x)}{\rm ind}_{Y}(x)-{1\over 2}\chi(Y){\rm dim}Y\right\}\\
\cdot \int_{1}^{+\infty}{{e^{-2t}}\over{1-e^{-2t}}}{{dt}\over{t}}\\
-{\rm rk}(\widetilde{F}){1\over 4}\chi(Y)\int_{1}^{+\infty}\left({{\sinh(2t)}\over{\cosh(2t)+1}}-1\right){{dt}\over{t}}\\
+{1\over 2}\left(\widetilde{\chi}'_{\phi}(\widetilde{F})-{m\over 2}\chi_{\phi}(\widetilde{F})\right)\log(Td)+{{a_{-1}}\over{\sqrt{d}}}-{\sqrt{T}}a_{-1}+o(1).
\end{multline}

Combining (\ref{3.38}), (\ref{3.481}) and (\ref{3.551})-(\ref{3.601}), one deduces, by setting $T\to +\infty$, that
\begin{multline}\label{3.611}
\log\left({{P^{{\rm det}H}_{\infty}\left(b^{\rm RS}_{{\rm det}(H^*(\widetilde{X},\widetilde{F}),\mathbb{Z}_{2})}\right)}\over{b^{M,\nabla \widetilde{f}}_{{\rm det}(H^*(\widetilde{X},\widetilde{F}),\mathbb{Z}_{2})}}}(\phi)\right)\simeq\\
-2{\rm Tr}_{s}^{B_{\partial}}[f]T+\left(\widetilde{\chi}'_{\phi}(\widetilde{F})-{m\over 2}\chi_{\phi}(\widetilde{F})\right)\log T-\left(\widetilde{\chi}'_{\phi}(\widetilde{F})-{m\over 2}\chi_{\phi}(\widetilde{F})\right)\log \pi\\
-\int_{Y}\theta(F,b^{F})(\nabla f)^*\psi \left(TY,\nabla^{TY}\right)\\
+2\sqrt{T}{\rm rk}(\widetilde{F})\int_{Y}\int^{B}L\exp(-B_{T})-2\sqrt{T}a_{-1}+2T{\rm rk}(\widetilde{F})\int_{Y}f\int^{B}\exp(-B_{T})\\
-2T{\rm rk}(\widetilde{F})\int_{Y}f\int^{B}\exp(-B_{0})\\
-{\rm rk}(\widetilde{F})\left\{\sum_{x\in B_{Y}}(-1)^{{\rm ind}_{Y}(x)}{\rm ind}_{Y}(x)-{1\over 2}\chi(Y){\rm dim}Y\right\}\\
\cdot \left(\int_{0}^{1}\left({{1+e^{-2t}}\over{1-e^{-2t}}}-{1\over t}\right){{dt}\over{t}}+\int_{1}^{+\infty}{{2e^{-2t}}\over{1-e^{-2t}}}{{dt}\over{t}}\right)\\
+2{\rm rk}(\widetilde{F}){1\over 4}\chi(Y)\cdot\left(\int_{0}^{1}{{\sinh(2t)}\over{\cosh(2t)+1}}{{dt}\over{t}}+\int_{1}^{+\infty}
\left({{\sinh(2t)}\over{\cosh(2t)+1}}-1\right){{dt}\over{t}}\right)\\
-\left(\widetilde{\chi}'_{\phi}(\widetilde{F})-{m\over 2}\chi_{\phi}(\widetilde{F})\right)\log(Td)-2{{a_{-1}}\over {\sqrt{d}}}+2\sqrt{T}a_{-1}\\
+2T{\rm rk}(\widetilde{F})\int_{Y}f e\left(TY,\nabla^{TY}\right)+{{2a_{-1}}\over{\sqrt{d}}}-(\Gamma'(1)-\log d)\left(\widetilde{\chi}'_{\phi}(\widetilde{F})-{m\over 2}\chi_{\phi}(\widetilde{F})\right)+o(1).
\end{multline}

By \cite[Theorem 3.20]{BZ1} and \cite[(7.72)]{BZ1}, one has
\begin{multline}\label{3.621}
\lim_{T\to +\infty}\left(2T{\rm rk}(\widetilde{F})\int_{Y}f\int^{B}\exp(-B_{T})-2T{\rm Tr}_{s}^{B_{\partial}}[f]\right)\\
=-{\rm rk}(\widetilde{F})\left\{\sum_{x\in B_{Y}}(-1)^{{\rm ind}_{Y}(x)}{\rm ind}_{Y}(x)-{1\over 2}\chi(Y){\rm dim}Y\right\},
\end{multline}
\begin{multline}\label{3.631}
\lim_{T\to +\infty}2\sqrt{T}{\rm rk}(\widetilde{F})\int_{Y}\int^{B}L\exp(-B_{T})\\
=2{\rm rk}(\widetilde{F})\left\{\sum_{x\in B_{Y}}(-1)^{{\rm ind}_{Y}(x)}{\rm ind}_{Y}(x)-{1\over 2}\chi(Y){\rm dim}Y\right\}.
\end{multline}

On the other hand, by \cite[(7.93)]{BZ1} and \cite[(5.55)]{BZ2}, one has
\begin{align}\label{3.641}
\int_{0}^{1}\left({{1+e^{-2t}}\over{1-e^{-2t}}}-{1\over t}\right)+\int_{1}^{+\infty}{{2e^{-2t}}\over{1-e^{-2t}}}{{dt}\over{t}}=1-\log\pi -\Gamma'(1),
\end{align}

\begin{multline}\label{3.651}
\int_{0}^{1}\left({{\sinh(2t)}\over{\cosh(2t)+1}}\right){{dt}\over{t}}+\int_{1}^{+\infty}
\left({{\sinh(2t)}\over{\cosh(2t)+1}}-1\right){{dt}\over{t}}\\
=-\log\pi-{{\Gamma'}\over{\Gamma}}\left({1\over 2}\right).
\end{multline}
From (\ref{3.611})-(\ref{3.651}), we get
\begin{multline}\label{3.41}
{{b^{\rm RS}_{{\rm det}(H^{\bullet}(\widetilde{X},\widetilde{F}),\mathbb{Z}_{2})}}\over{b^{M,\nabla f}_{{\rm det}
(H^{\bullet}(\widetilde{X},\widetilde{F}),\mathbb{Z}_{2})}}}(\phi)=\exp\left(-\int_{Y}\theta(\widetilde{F},
b^{\widetilde{F}})(\nabla f)^*\psi(TY,\nabla^{TY})\right.\\
\left.-{1\over 4}\sum_{x\in B_{\partial}}{\rm rk}(F)(-1)^{{\rm ind}_{Y}(x)}\left(2{\Gamma'\over\Gamma}\left({1\over 2}\right)-2\Gamma'(1)\right)\right).
\end{multline}
By \cite[(5.53)]{BZ2}, we know
\begin{align}\label{3.42}
{{\Gamma'}\over{\Gamma}}\left({1\over 2}\right)-\Gamma'(1)=-2\log(2).
\end{align}
Then from (\ref{3.41}), (\ref{3.42}) and Lemma \ref{t1.1}, we get (\ref{3.40}).
\end{proof}

\subsection{Proofs of the intermediate results}

The purpose of this subsection is to prove the intermediate results. Since the methods of the proofs of these
theorems are essentially the same as the corresponding theorems in \cite{SZ}, so we refer to \cite{SZ} for related
definitions and notations directly when there will be no confusion, such as $B_{b,g}$, $A_{b,t,T}$, $A_{g,t,T}$,
$C_{t,T}$, $\cdots$.
\subsubsection{Proof of Theorem \ref{t3.4}}
First, as it in \cite[(4.45)]{SZ}, we have
\begin{multline}\label{3.43}
{{P_{T}^{[0,1],{\rm det}H}(b^{\rm RS}_{{\rm det}(H^{\bullet}(\Omega^*_{[0,1],T}(\widetilde{X},\widetilde{F})),\mathbb{Z}_{2})})}\over{b^{M,\nabla f}_{{\rm det}(H^{\bullet}(\widetilde{X},\widetilde{F}),\mathbb{Z}_{2})}}}(\phi)\\
=\prod_{i=0}^{m}{\rm det}\left(P_{\infty,T}^{\#}P_{\infty,T}|_{\Omega^{i}_{[0,1],T}(\widetilde{X},\widetilde{F})}\right)^{(-1)^{i+1}}(\phi).
\end{multline}

From \cite[Propositions 4.4 and 4.5]{SZ}, one deduces that as $T\to +\infty$,
\begin{multline}\label{3.44}
{\rm det}\left(P_{\infty,T}^{\#}P_{\infty,T}|_{\Omega^{i}_{[0,1],T}(\widetilde{X},\widetilde{F})}\right)(\phi)\\
={\rm det}\left(e_{T}e_{T}^{\#}P_{\infty,T}^{\#}P_{\infty,T}|_{\Omega^{i}_{[0,1],T}(\widetilde{X},\widetilde{F})}\right)(\phi)
\cdot {\rm det}^{-1}\left(e_{T}e_{T}^{\#}|_{\Omega^{i}_{[0,1],T}(\widetilde{X},\widetilde{F})}\right)(\phi)\\
={\rm det}\left((P_{\infty,T}e_{T})^{\#}P_{\infty,T}e_{T}|_{C^{i}(W^{u},\widetilde{F})}\right)(\phi)\cdot {\rm det}^{-1}\left(
e_{T}^{\#}e_{T}|_{C^{i}(W^{u},\widetilde{F})}\right)(\phi)\\
={\rm det}\left((1+O(e^{-cT}))^{\#}\left({\pi\over T}\right)^{N-m/2}e^{2T\mathcal{F}}(1+O(e^{-cT}))|_{C^{i}(W^{u},\widetilde{F})}\right)(\phi)
\end{multline}

From (\ref{3.43}) and (\ref{3.44}), we get the result immediately.

\subsubsection{Proof of Theorem \ref{t3.5}}
The proof of Theorem \ref{t3.5} is the same as the proof of \cite[Theorem 3.4]{SZ} given in \cite[Section 5]{SZ}.

\subsubsection{Proof of Theorem \ref{t3.6}}
Recall that the operator $e_{T}:C^*(W^{u},\widetilde{F})\to \Omega^*_{[0,1],T}(\widetilde{X},\widetilde{F})$ has been
defined in \cite[(4.38)]{SZ}, and in the current case, we also have that $e_{T}$ commutes with $\mathbb{Z}_{2}$. So by
\cite[Proposition 4.4]{SZ}, we have that for $T\geq 0$ large enough, $e_{T}:C^*(W^{u},\widetilde{F})\to \Omega^*_{[0,1],T}(\widetilde{X},\widetilde{F})$ is an identification of $\mathbb{Z}_{2}$-spaces. So (\ref{3.34}) follows. Also (\ref{3.35}) is from \cite[proposition 4.2]{SZ}.

\subsubsection{Proof of Theorem \ref{t3.7}}
In this section, we provide a proof of Theorem \ref{t3.7}, which computes the asymptotic of ${\rm Tr}_{s}[gN\exp(-tD^{2}_{b_{T}})]$ for fixed $T\geq 0$ as $t\to 0$. The method is essentially the same as it in \cite{SZ}.

By \cite[(6.4)]{SZ}, we have
\begin{multline}\label{3.45}
e^{-tD_{b}^{2}}=e^{-tD_{g}^{2}}
+\sum_{k=1}^{m}(-1)^{k}t^{k}\int_{\Delta_{k}}e^{-t_{1}tD_{g}^{2}}B_{b,g}e^{-t_{2}tD_{g}^{2}}\cdots B_{b,g}e^{-t_{k+1}t
D_{g}^{2}}dt_{1}\cdots dt_{k}\\
+(-1)^{m+1}t^{m+1}\int_{\Delta_{m+1}}e^{-t_{1}tD_{g}^{2}}B_{b,g}e^{-t_{2}tD_{g}^{2}}\cdots B_{b,g}e^{-t_{m+2}tD_{b}^{2}}dt_{1}\cdots dt_{m+1},
\end{multline}
where $\Delta_{k}$, $1\leq k\leq m+1$, is the $k$-simplex defined by $t_{1}+\cdots +t_{k+1}=1$, $t_{1}\geq 0$, $\cdots$, $t_{k+1}\geq 0$. Also, by the same proof of \cite[Proposition 6.1]{SZ}, we have the following result.

\begin{proposition}
As $t\to 0^{+}$, one has
\begin{align}\label{3.46}
t^{m+1}\int_{\Delta_{m+1}}{\rm Tr}_{s}\left[\phi Ne^{-t_{1}tD_{g}^{2}}B_{b,g}e^{-t_{2}tD_{g}^{2}}\cdots B_{b,g}e^{-t_{m+2}tD_{b}^{2}}\right]dt_{1}\cdots dt_{m+1}\to 0.
\end{align}
\end{proposition}

Now for $1\leq k\leq m$, we want to prove
\begin{align}
\lim_{t\to 0^{+}}t^{k}\int_{\Delta_{k}}{\rm Tr}_{s}\left[\phi Ne^{-t_{1}tD_{g}^{2}}B_{b,g}e^{-t_{2}tD_{g}^{2}}\cdots B_{b,g}e^{-t_{k+1}tD_{g}^{2}}\right]dt_{1}\cdots dt_{k}=0.
\end{align}

First from \cite[Theorem 3.3]{DZ}, in our case, we have, as $t\to 0^{+}$,
\begin{multline}\label{3.801}
t^{k}\int_{\Delta_{k}}{\rm Tr}_{s}\left[\phi Ne^{-t_{1}tD_{g}^{2}}B_{b,g}e^{-t_{2}tD_{g}^{2}}\cdots B_{b,g}e^{-t_{k+1}tD_{g}^{2}}\right]dt_{1}\cdots dt_{k}\\
=\sum_{|\lambda(k)|\leq m-k}{{(-1)^{|\lambda(k)|}}\over{\lambda(k)!\widetilde{\lambda}(k)!}}{\rm Tr}_{s}\left[\phi\mathcal{D}_{t}^{\lambda(k)}\exp(-t D^{2}_{g})\right]+O(\sqrt{t}),
\end{multline}
where
$$\mathcal{D}_{t}^{\lambda(k)}=t^{k+|\lambda(k)|} N B^{[\lambda_{1}]}B^{[\lambda_{2}]}\cdots B^{[\lambda_{k}]},$$
$$B^{[0]}=B_{b,g},\ \ B^{[k]}=\left[D^{2}_{g},B^{[k-1]}\right],$$
and we refer to \cite[(3.22)-(3.26)]{DZ} the other notations.

Let $N_{Y}=TX/TY$ be the normal bundle to $Y$ in $X$. We identity $N_{Y}$ with the orthogonal bundle to $TY$ in $TX$. By standard estimates of heat kernel, the problem in calculating $t\to 0^{+}$ in (\ref{3.801}) can be localized to an open neighborhood $\mathcal{U}_{\varepsilon}$ of $Y$ in $X$. Using normal geodesic coordinates to $Y$ in $X$, we will identity $\mathcal{U}_{\varepsilon}$ to an $\varepsilon$-neighborhood of $Y$ in $N_{Y}$.

Since we have used normal geodesic coordinates to $Y$ in $X$, if $(y,z)\in N_{Y}$,
\begin{align}
\phi^{-1}(y,z)=(y,\phi^{-1}z).
\end{align}

Let $dv_{Y}$, $dv_{N_{Y}}$ be the Riemannian volumes on $TY$, $N_{Y}$ induced by $g^{T\widetilde{X}}$. Let $k(y,z)$ $(y\in Y, z\in N_{Y}, |z|<\varepsilon)$ be defined by
\begin{align}
dv_{X}=k(y,z)dv_{Y}(y)dv_{N_{Y}}(z).
\end{align}
Then
\begin{align}
k(y,0)=1.
\end{align}

Let $\rho(Z)$ be a smooth function which is equal to $1$ if $|Z|\leq {1\over 4}\varepsilon$ and equal to $0$ if $|Z|\geq {1\over 2}\varepsilon$. Take $x_{0}\in Y$, let $\mathbb{F}_{x_{0}}$ be the smooth sections of $(\Lambda(T^*\widetilde{X})\otimes \widetilde{F})_{x_{0}}$ over $T_{x_{0}}\widetilde{X}$. Let $\Delta^{T\widetilde{X}}$ be the standard Laplacian on $T_{x_{0}}\widetilde{X}$ with respect to the metric $g^{T\widetilde{X}}$.

Let $J_{t}^{1}$ be the operator acting on $\mathbb{F}_{x_{0}}$
\begin{align}\label{3.84}
J_{t}^{1}=(1-\rho^{2}(Z))(-t\Delta^{T_{x_{0}}\widetilde{X}})+\rho^{2}(Z)tD_{g}^{2}.
\end{align}

Let $H_{t}$ be the linear map
\begin{align}
s(Z)\in \mathbb{F}_{x_{0}}\to s\left({Z\over{\sqrt{t}}}\right)\in \mathbb{F}_{x_{0}}.
\end{align}
Set
\begin{align}\label{3.86}
J_{t}^{2}=H_{t}^{-1}J_{t}^{1}H_{t}.
\end{align}

Let $e_{1},\cdots, e_{m-1}$ be an oriented orthonormal base of $T_{x_{0}}Y$ and let $e_{m}$ be an orthonormal base of $N_{Y}$.

Let $J_{t}^{3}$ be the operator obtained from $J_{t}^{2}$ by replacing $c(e_{i})$, $\widehat{c}(e_{i})$, $1\leq i\leq m-1$ by
\begin{align}\label{3.79}
c_{t}(e_{i})={{e_{i}}\over{t^{1\over 4}}}\wedge -t^{1\over 4}i_{e_{i}},\ \widehat{c}_{t}(e_{i})={{\widehat{e}_{i}}\over{t^{1\over 4}}}\wedge +{t^{1\over 4}}i_{\widehat{e}_{i}},\ 1\leq i\leq m-1.
\end{align}

Let $G_{t}$ be the process of (\ref{3.86}) and (\ref{3.79}) in the above. Let $P_{t}$ be the smooth kernel of $\exp(-tD^{2}_{g})$ and let $P^{i}_{t}(z,z')$ $(z,z'\in T_{x_{0}}\widetilde{X}, i=1,2,3)$ be the smooth kernel associated to $\exp(-J_{t}^{i})$ with respect to $dv_{T_{x_{0}}\widetilde{X}}(z')$. Then we have

\begin{multline}\label{3.881}
\lim_{t\to 0^{+}}{\rm Tr}_{s}\left[\phi\mathcal{D}_{t}^{\lambda(k)}\exp\left(-tD^{2}_{g}\right)\right]
=\lim_{t\to 0^{+}}\int_{\mathcal{U}_{\varepsilon}/8}{\rm Tr}_{s}\left[\phi \mathcal{D}_{t}^{\lambda(k)}P_{t}(\phi^{-1}x,x)\right]dv_{X}(x)
=\\ \lim_{t\to 0^{+}}\int_{y\in Y}\int_{z\in N_{Y},|z|\leq \varepsilon/8}{\rm Tr}_{s}\left[\phi\mathcal{D}_{t}^{\lambda(k)}P_{t}(\phi^{-1}(y,z),(y,z))\right]k(y,z)dv_{Y}(y)dv_{N_{Y}}(z).
\end{multline}

By (\ref{3.84}) and the finite propagation speed, there exist $c,C>0$ such that for $z\in N_{Y}$, $|z|\leq {1\over 8}\varepsilon$, $0<t\leq 1$, we have
\begin{align}\label{3.891}
\left|P_{t}(\phi^{-1}(y,z),(y,z))k(y,z)-
P_{t}^{1}(\phi^{-1}z,z)\right|\leq c\exp\left(-{C\over{t^{2}}}\right).
\end{align}

Let $[G_{t}\left(\mathcal{D}_{t}^{\lambda(k)}\right)P_{t}^{3}
\left({\phi^{-1}z},z\right)]^{\rm max}\in {\rm End}(\Lambda^*(N_{Y}))\otimes {\rm End}\widetilde{F}$ be the coefficient of $e^{1}\wedge \cdots e^{m-1}\wedge \widehat{e}^{1}\cdots \wedge\cdots \wedge \widehat{e}^{m-1}$ in the expansion of it. Then by \cite[Proposition 4.11]{BZ1}, we have
\begin{multline}
{\rm Tr}_{s}\left[\phi \mathcal{D}_{t}^{\lambda(k)}P_{t}^{1}(\phi^{-1}z,z)\right]\\
=2^{m-1}(-1)^{{(m-1)m}\over{2}}{1\over {\sqrt{t}}}{\rm Tr}_{s}\left[\phi\left[ G_{t}\left(\mathcal{D}_{t}^{\lambda(k)}\right)P_{t}^{3}
\left({{\phi^{-1}z}\over{\sqrt{t}}},{z\over{\sqrt{t}}}\right)\right]^{\rm max}\right].
\end{multline}
Then
\begin{multline}\label{3.901}
\lim_{t\to 0^{+}}\int_{z\in N_{Y},|z|\leq \varepsilon/8}{\rm Tr}_{s}\left[\phi \mathcal{D}_{t}^{\lambda(k)}P_{t}^{1}(\phi^{-1}z,z)\right]dv_{N_{Y}}(z)=\lim_{t\to 0^{+}}\int_{z\in N_{Y},|z|\leq \varepsilon/8}\\
2^{m-1}(-1)^{{(m-1)m}\over{2}}{1\over {\sqrt{t}}}{\rm Tr}_{s}\left[\phi\left[ G_{t}\left(\mathcal{D}_{t}^{\lambda(k)}\right)P_{t}^{3}
\left({{\phi^{-1}z}\over{\sqrt{t}}},{z\over{\sqrt{t}}}\right)\right]^{\rm max}\right]dv_{N_{Y}}(z).
\end{multline}

Let $a>0$ be the injectivity radius of $(\widetilde{X},g^{T\widetilde{X}})$. We identify the open ball $B^{T_{x\widetilde{X}}}(0,{a\over 2})$ with the open ball $B^{\widetilde{X}}(x,{a\over 2})$ in $\widetilde{X}$ using geodesic coordinates. Then $y\in T_{x}\widetilde{X}$, $|y|\leq {a\over 2}$, represents an element of $B^{\widetilde{X}}(0,{a\over 2})$. For $y\in T_{x}\widetilde{X}$, $|y|\leq {a\over 2}$, we identify $T_{y}\widetilde{X}$, $\widetilde{F}_{y}$ to $T_{x}\widetilde{X}$, $\widetilde{F}_{x}$ by parallel transport along the geodesic $t\in [0,1]\to ty$ with respect to the connections $\nabla^{T\widetilde{X}}$, $\nabla^{\widetilde{F},u}$ respectively.

Let $\Gamma^{T\widetilde{X},x}$, $\Gamma^{\widetilde{F},u,x}$ be the connection forms for $\nabla^{T\widetilde{X}}$, $\nabla^{\widetilde{F},u}$ in the considered trivialization of $T\widetilde{X}$. By \cite[Proposition 4.7]{ABP}, one has
$$\Gamma^{T\widetilde{X},x}={1\over 2}R_{x}^{T\widetilde{X}}(y,\cdot)+O(|y|^{2}),$$
\begin{align}
\Gamma^{\widetilde{F},u,x}_{y}=O(|y|).
\end{align}

Then by direct computation, we find that as $t\to 0^{+}$,
\begin{align}\label{3.80}
G_{t}\left(t^{l+1}B^{[l]}\right)=O(\sqrt{t}),\ \ l\geq 0,
\end{align}
and
\begin{align}\label{3.81}
G_{t}(N)={1\over {2\sqrt{t}}}\sum_{i=1}^{m-1}e_{i}\wedge \widehat{e}_{i}+O(1)={1\over {\sqrt{t}}}L|_{Y}+O(1),
\end{align}
then
\begin{align}\label{3.941}
\lim_{t\to 0^{+}}G_{t}\left(\mathcal{D}_{t}^{\lambda(k)}\right) {\rm exists\ and}\ \lim_{t\to 0^{+}}G_{t}\left(\mathcal{D}_{t}^{\lambda(k)}\right)=0,\ 1<k\leq m.
\end{align}

Using \cite[(4.29)]{BZ1}, one can find that as $t\to 0^{+}$,
\begin{align}\label{3.94}
J_{t}^{3}\to J_{0}^{3}=-\Delta^{T_{x_{0}}\widetilde{X}}+{1\over 2}\dot{R}^{TY}.
\end{align}

Then by (\ref{3.901}), (\ref{3.80}), (\ref{3.81}) and (\ref{3.94}), we have
\begin{multline}\label{3.97}
\lim_{t\to 0^{+}}\int_{z\in N_{Y},|z|\leq \varepsilon/8}{\rm Tr}_{s}\left[\phi \mathcal{D}_{t}^{\lambda(k)}P_{t}^{1}(\phi^{-1}z,z)\right]dv_{N_{Y}}(z)\\
=2^{m-1}(-1)^{{(m-1)m}\over 2}\int_{z\in N_{Y}}{\rm Tr}_{s}\left[\phi\left[ G_{0}\left(\mathcal{D}_{0}^{\lambda(k)}\right)P_{0}^{3}(\phi^{-1}z,z)\right]^{\rm max}\right]dv_{N_{Y}}(z).
\end{multline}

Then by (\ref{3.801}), (\ref{3.881}), (\ref{3.891}), (\ref{3.941}), (\ref{3.94}) and (\ref{3.97}), we have that for any $1<k\leq m$,
\begin{align}\label{3.47}
\lim_{t\to 0^{+}}t^{k}\int_{\Delta_{k}}{\rm Tr}_{s}\left[\phi Ne^{-t_{1}tD_{g}^{2}}B_{b,g}e^{-t_{2}tD_{g}^{2}}\cdots B_{b,g}e^{-t_{k+1}tD_{g}^{2}}\right]dt_{1}\cdots dt_{k}=0,
\end{align}
while for $k=1$, $0\leq t_{1}\leq 1$,
using the standard heat kernel on $\mathbb{R}^{n}$, $\phi^{-1}z=-z$, \cite[(6.16)]{SZ} and
$${1\over{2\sqrt{\pi t}}}\int_{\mathbb{R}}\exp\left(-{{4|y|^{2}}\over{4t}}\right)dy={1\over 2},\ \ {\rm Tr}_{s}\left[c(e_{m})\widehat{c}(e_{m})\right]=-2,$$
we have
\begin{multline}\label{3.48}
\lim_{t\to 0^{+}}t{\rm Tr}_{s}\left[\phi Ne^{-t_{1}tD_{g}^{2}}B_{b,g}e^{-(1-t_{1})tD_{g}^{2}}\right]=\lim_{t\to 0^{+}}t{\rm Tr}_{s}\left[\phi NB_{b,g}e^{-tD^{2}_{g}}\right]\\
={1\over 2}\int_{Y}\int^{B}{\rm Tr}\left[\phi \left(\sum_{i,j=1}^{m-1}e_{i}\wedge \widehat{e}_{j}(\nabla^{u}_{e_{i}}\omega^{F}(e_{j}))+{1\over 2}\left[\omega^{F},\widehat{\omega^{F}_{g}}-\widehat{\omega^{F}}\right]\right)\right]\\
\cdot L\exp\left(-{{\dot{R}^{TY}}\over{2}}\right).
\end{multline}
So by \cite[(2.13)]{BZ2}, and proceeding as in \cite[(6.26)-(6.28)]{SZ}, we have
\begin{align}\label{3.49}
\lim_{t\to 0^{+}}t{\rm Tr}_{s}\left[\phi Ne^{-t_{1}tD_{g}^{2}}B_{b,g}e^{-(1-t_{1})tD_{g}^{2}}\right]=0.
\end{align}

From (\ref{3.45}), (\ref{3.46}), (\ref{3.47}) and (\ref{3.49}) and \cite[Theorem 5.9]{BZ2}, one can get the result.

\subsubsection{Proof of Theorem \ref{t3.9}}

In order to prove (\ref{t3.9}), one needs only to prove that under the, conditions of Theorem \ref{t3.9}, there exists a constant $C">0$ such that
\begin{multline}\label{3.50}
\left|{\rm Tr}_{s}\left[\phi N\exp\left(-(tD_{b}+T\widehat{c}(\nabla f))^{2}\right)\right]-{\rm Tr}_{s}\left[\phi N\exp\left(-(tD_{g}+T\widehat{c}(\nabla f))^{2}\right)\right]\right.\\
\left.-{T\over 2}\int_{Y}(\theta_{\phi}(\widetilde{F},b^{\widetilde{F}})-\theta_{\phi}(\widetilde{F},g^{\widetilde{F}}))\int^{B}
\widehat{df}\exp(-B_{T^{2}})\right|\leq C"t.
\end{multline}

By \cite[(7.8)]{SZ}, we have
\begin{multline}\label{3.51}
e^{-A^{2}_{b,t,T}}=e^{-A^{2}_{g,t,T}}\\
+\sum_{k=1}^{m}(-1)^{k}\int_{\Delta_{k}}e^{-t_{1}A^{2}_{g,t,T}}C_{t,T}e^{-t_{2}A^{2}_{g,t,T}}\cdots C_{t,T}e^{-t_{k+1}A^{2}_{g,t,T}}dt_{1}\cdots dt_{k}\\
+(-1)^{m+1}\int_{\Delta_{m+1}}e^{-t_{1}A^{2}_{g,t,T}}C_{t,T}e^{-t_{2}A^{2}_{g,t,T}}\cdots C_{t,T}e^{-t_{m+2}A^{2}_{b,t,T}}dt_{1}\cdots dt_{m+1}.
\end{multline}
From the proof of \cite[(7.21)]{SZ}, we have that there exists $C_{1}>0$ such that for any $t>0$ small enough and $T\in[0,{1\over t}]$,
\begin{multline}\label{3.52}
\left|\int_{\Delta_{m+1}}{\rm Tr}_{s}\left[\phi Ne^{-t_{1}A^{2}_{g,t,T}}C_{t,T}e^{-t_{2}A^{2}_{g,t,T}}\cdots C_{t,T}e^{-t_{m+2}A^{2}_{b,t,T}}\right]dt_{1}\cdots dt_{m+1}\right|\\
\leq C_{1}t
\end{multline}

By the standard heat kernel expansion, we see that for $1\leq k\leq m$, our problem can be localized near $Y$.

Now for any $x\in Y$, let $e_{1},\cdots, e_{m-1},e_{m}$ be a orthonormal basis of $T\widetilde{X}|_{Y}$ such that $e_{1},\cdots,e_{m-1}$ is an orthonormal basis of $TY$ and $e_{m}$ is the normal vector field along $Y$. Then we use the Getzler rescaling (cf. \cite{BGV}, \cite{Ge1}, \cite{Ge2}) introduced in (\ref{3.79}), with $t$ there replaced by $t^{2}$ here. By using \cite[(7.7)]{SZ}, one has
\begin{align}\label{3.88}
G_{t^{2}}(C_{t,T})=G_{t^{2}}(t^{2}B_{b,g})+tT\omega^{F}(\nabla f).
\end{align}
By (\ref{3.81}), we have
\begin{align}\label{3.89}
G_{t^{2}}(N)={1\over{2t}}\sum_{i=1}^{m-1}e^{i}\wedge \widehat{e}^{i}+O(1)={1\over t}L|_{Y}+O(1).
\end{align}

By (\ref{3.88}), (\ref{3.89}), replacing (\ref{3.84}) by \cite[(13.8)]{BZ1}, that is,
\begin{align}
J_{t}^{1}=\left(1-\rho^{2}(Z)\right)\left(-t^{2}\Delta^{T_{x_{0}}\widetilde{X}}+T^{2}\right)+\rho^{2}(Z)\left(tD_{g}+
T\widehat{c}(\nabla f)\right)^{2},
\end{align}
using $\mathbb{Z}_{2}$-equivariant version of \cite[Proposition 13.3]{BZ1} (cf. \cite[Proposition 11.5]{B}) and applying the steps (\ref{3.801})-(\ref{3.49}), we have that there exists $C_{2}>0$, $0<d<1$ such that for any $1<k\leq m$,
$0<t\leq d$, $T\geq 0$ with $tT\leq 1$,
\begin{multline}\label{3.53}
\left|\int_{\Delta_{k}}{\rm Tr}_{s}\left[\phi Ne^{-t_{1}A^{2}_{g,t,T}}C_{t,T}e^{-t_{2}A^{2}_{g,t,T}}\cdots C_{t,T}e^{-t_{k+1}A^{2}_{g,t,T}}\right]dt_{1}\cdots dt_{k}\right|\\
\leq C_{2}t.
\end{multline}

For $k=1$, we have
\begin{align}\label{3.90}
{\rm Tr}_{s}\left[\phi Ne^{-t_{1}A^{2}_{g,t,T}}C_{t,T}e^{-(1-t_{1})A^{2}_{g,t,T}}\right]={\rm Tr}_{s}\left[\phi NC_{t,T}e^{-A^{2}_{g,t,T}}\right].
\end{align}
From (\ref{3.88}), (\ref{3.89}), (\ref{3.90}) and proceeding as above, one has for any $0<t\leq d$, $T\geq 0$ with $tT\leq 1$ and $0\leq t_{1}\leq 1$,

\begin{multline}\label{3.54}
\left|{\rm Tr}_{s}\left[\phi Ne^{-t_{1}A^{2}_{g,t,T}}C_{t,T}e^{-(1-t_{1})A^{2}_{g,t,T}}\right]-T\int_{Y}\int^{B}{\rm Tr}\left[\phi \omega^{F}(\nabla f)\right]L\exp(-B_{T^{2}})\right|\\
\leq C_{2}t.
\end{multline}

Now similar as \cite[(7.25)]{SZ}, we have
\begin{multline}\label{3.55}
\int_{Y}\int^{B}{\rm Tr}\left[\phi\omega^{F}(\nabla f)\right]L\exp(-B_{T^{2}})\\
={1\over 2}\int_{Y}\left(\theta_{\phi}(\widetilde{F},g^{\widetilde{F}})-\theta_{\phi}(\widetilde{F},b^{\widetilde{F}})\right)
\int^{B}\widehat{\nabla f}\exp(-B_{T^{2}}).
\end{multline}

From (\ref{3.51})-(\ref{3.55}), we get (\ref{3.50}), which completes the proof.

\subsubsection{Proof of Theorem \ref{t3.10}}

In order to prove Theorem \ref{t3.10}, we need only to prove that for any $T>0$,
\begin{align}\label{3.56}
\lim_{t\to 0^{+}}\left({\rm Tr}_{s}\left[\phi N\exp\left(-A^{2}_{b,t,{T\over t}}\right)\right]-{\rm Tr}_{s}\left[\phi N\exp\left(-A^{2}_{g,t,{T\over t}}\right)\right]\right)=0.
\end{align}

By \cite[(8.2) and (8.4)]{SZ}, there exists $0<C_{0}\leq 1$, such that when $0<t\leq C_{0}$, one has the absolute convergent expansion formula
\begin{multline}\label{3.57}
e^{-A^{2}_{b,t,{T\over t}}}-e^{-A^{2}_{g,t,{T\over t}}}\\
=\sum_{k=1}^{+\infty}(-1)^{k}\int_{\Delta_{k}}e^{-t_{1}A^{2}_{g,t,{T\over t}}}C_{t,{T\over t}}e^{-t_{2}A^{2}_{g,t,{T\over t}}}\cdots C_{t,{T\over t}}e^{-t_{k+1}A^{2}_{g,t,{T\over t}}}dt_{1}\cdots dt_{k},
\end{multline}
and that
\begin{align}\label{3.58}
\sum_{k=m}^{+\infty}(-1)^{k}\int_{\Delta_{k}}\phi e^{-t_{1}A^{2}_{g,t,{T\over t}}}C_{t,{T\over t}}e^{-t_{2}A^{2}_{g,t,{T\over t}}}\cdots C_{t,{T\over t}}e^{-t_{k+1}A^{2}_{g,t,{T\over t}}}dt_{1}\cdots dt_{k}
\end{align}
is uniformly absolute convergent for $0<t\leq C_{0}$.

Proceeding as in \cite[Section 8]{SZ}, for any $(t_{1},\cdots,t_{k+1})\in \Delta_{k}\backslash \{t_{1}\cdots t_{k+1}=0\}$, one has that
\begin{multline}\label{3.59}
\left|{\rm Tr}_{s}\left[\phi Ne^{-t_{1}A^{2}_{g,t,{T\over t}}}C_{t,{T\over t}}e^{-t_{2}A^{2}_{g,t,{T\over t}}}\cdots C_{t,{T\over t}}e^{-t_{k+1}A^{2}_{g,t,{T\over t}}}\right]\right|\\
\leq C_{3}t^{k}(t_{1}\cdots t_{k})^{-{1\over 2}}{\rm Tr}\left[e^{-{{A^{2}_{g,t,{T\over t}}}\over{2}}}\right]
\left\|\psi e^{-{{t_{k+1}}\over{2}}A^{2}_{g,t,{T\over t}}}\right\|
\end{multline}
for some positive constant $C_{3}>0$.

Also by \cite[(8.4)]{SZ}, (\ref{3.59}) and the same assumption in \cite{SZ} that $t_{k+1}\geq {1\over{}k+1}$, one gets
\begin{multline}\label{3.60}
\left|\int_{\Delta_{k}}{\rm Tr}_{s}\left[\phi Ne^{-t_{1}A^{2}_{g,t,{T\over t}}}C_{t,{T\over t}}e^{-t_{2}A^{2}_{g,t,{T\over t}}}\cdots C_{t,{T\over t}}e^{-t_{k+1}A^{2}_{g,t,{T\over t}}}\right]dt_{1}\cdots dt_{k}\right|\\
\leq C_{4}t^{k-m}\left\|\psi e^{-{1\over{2(k+1)}}A^{2}_{g,t,{T\over t}}}\right\|
\end{multline}
for some constant $C_{4}>0$.

From (\ref{3.57}), (\ref{3.58}), (\ref{3.60}), \cite[(8.9) and (8.10)]{SZ} and the dominate convergence, we get
(\ref{3.56}).

\begin{remark}\label{3.17}
The right hand side of (\ref{3.520}) is not stated in \cite{BZ2}, so we explain it in our case. First by direct computation, it equals
\begin{align}\label{3.116}
{1\over{1-e^{-2T}}}\left(\left(1+e^{-2T}\right)\widetilde{\chi}'_{\phi}(\widetilde{F})-e^{-2T}(m-1)
\chi_{\phi}(\widetilde{F})\right)+{{e^{-T}}\over{e^{T}+e^{-T}}}\chi_{\phi}(\widetilde{F}).
\end{align}
Since near $x\in B_{Y}$, $\phi=1$ in $T_{x}Y$ and $\phi=-1$ in $N_{Y}$. So the first term of (\ref{3.116}) is just from \cite[Theorem A.2]{BZ2} corresponding to $\phi=1$. On $N_{Y}$, by \cite[(8.15)]{BZ2}, as the proof of \cite[Theorem 5.12]{BZ2}, we need to compute
\begin{multline}\label{3.117}
\int_{y\in N_{Y},|y|\leq \varepsilon}\left({{Te^{2T}}\over{2\pi t^{2}\sinh (2T)}}\right)^{1\over 2}\exp\left\{
-{{T(\cosh(2T)+1)}\over{t^{2}\sinh(2T)}}y^{2}\right\}dy\\
\times {\rm Tr}^{\Lambda(N_{Y})}_{s}\left[-N\exp[-2T(N^{+}+{\rm ind}_{N_{Y}}(x)-N^{-})]\right]\chi_{\phi}(\widetilde{F}).
\end{multline}
Since ${\rm ind}_{N_{Y}}(x)=0$, $N^{-}=0$, then as $t\to 0^{+}$, (\ref{3.117}) equals
$${{e^{T}}\over{e^{T}+e^{-T}}}\times e^{-2T}\chi_{\phi}(\widetilde{F}),$$
which is equal to the second term in (\ref{3.116}).
\end{remark}

\subsubsection{Proof of Theorem \ref{t3.11}}

In order to prove Theorem \ref{t3.11}, we need only to prove that there exist $c>0$, $C>0$, $0<C_{0}\leq 1$ such that
for any $0<t\leq C_{0}$, $T\geq 1$,
\begin{align}\label{3.61}
\left|{\rm Tr}_{s}\left[\phi N\exp\left(-A^{2}_{b,t,{T\over t}}\right)\right]-{\rm Tr}_{s}\left[\phi N\exp\left(-A^{2}_{g,t,{T\over t}}\right)\right]\right|\leq c\exp(-CT).
\end{align}

First of all, one can choose $C_{0}>0$ small enough so that for any $0<t\leq C_{0}$, $T>0$, by (\ref{3.57}), we have the absolute convergent expansion formula
\begin{multline}
e^{-A^{2}_{b,t,{T\over t}}}-e^{-A^{2}_{g,t,{T\over t}}}\\
=\sum_{k=1}^{+\infty}(-1)^{k}\int_{\Delta_{k}}e^{-t_{1}A^{2}_{g,t,{T\over t}}}C_{t,{T\over t}}e^{-t_{2}A^{2}_{g,t,{T\over t}}}\cdots C_{t,{T\over t}}e^{-t_{k+1}A^{2}_{g,t,{T\over t}}}dt_{1}\cdots dt_{k}
\end{multline}
from which one has
\begin{multline}\label{3.63}
{\rm Tr}_{s}\left[\phi N\exp\left(-A^{2}_{b,t,{T\over t}}\right)\right]-{\rm Tr}_{s}\left[\phi N\exp\left(-A^{2}_{g,t,{T\over t}}\right)\right]
=\\
\sum_{k=1}^{+\infty}(-1)^{k}\int_{\Delta_{k}}{\rm Tr}_{s}\left[\phi Ne^{-t_{1}A^{2}_{g,t,{T\over t}}}C_{t,{T\over t}}e^{-t_{2}A^{2}_{g,t,{T\over t}}}\cdots C_{t,{T\over t}}e^{-t_{k+1}A^{2}_{g,t,{T\over t}}}\right]dt_{1}\cdots dt_{k}.
\end{multline}
Thus, in order to prove (\ref{3.61}), we need only to prove
\begin{multline}
\sum_{k=1}^{+\infty}\left|\int_{\Delta_{k}}{\rm Tr}_{s}\left[\phi Ne^{-t_{1}A^{2}_{g,t,{T\over t}}}C_{t,{T\over t}}e^{-t_{2}A^{2}_{g,t,{T\over t}}}\cdots C_{t,{T\over t}}e^{-t_{k+1}A^{2}_{g,t,{T\over t}}}\right]dt_{1}\cdots dt_{k}\right|\\
=\sum_{k=1}^{+\infty}\left|\int_{\Delta_{k}}{\rm Tr}_{s}\left[\phi Ne^{-(t_{1}+t_{k+1})A^{2}_{g,t,{T\over t}}}C_{t,{T\over t}}e^{-t_{2}A^{2}_{g,t,{T\over t}}}\cdots C_{t,{T\over t}}\right]dt_{1}\cdots dt_{k}\right|\\
\leq c\exp(-CT).
\end{multline}

By \cite[(8.6)]{SZ}, we have for any $t>0$, $T\geq 1$, $(t_{1},\cdots, t_{k+1})\in\Delta_{k}\backslash\{t_{1}\cdots t_{k+1}=0\}$,
\begin{multline}\label{3.65}
{\rm Tr}_{s}\left[\phi Ne^{-(t_{1}+t_{k+1})A^{2}_{g,t,{T\over t}}}C_{t,{T\over t}}e^{-t_{2}A^{2}_{g,t,{T\over t}}}\cdots C_{t,{T\over t}}\right]\\
={\rm Tr}_{s}\left[\phi N\psi e^{-(t_{1}+t_{k+1})A^{2}_{g,t,{T\over t}}}C_{t,{T\over t}}\psi e^{-t_{2}A^{2}_{g,t,{T\over t}}}\cdots \psi e^{-t_{k}A^{2}_{g,t,{T\over t}}}C_{t,{T\over t}}\right].
\end{multline}

From (\ref{3.65}), \cite[(9.18) and (9.19)]{SZ}, one sees that there exists $C_{5}>0$, $C_{6}>0$ and $C_{7}>0$ such that for any $k\geq 1$,
\begin{multline}
\left|\int_{\Delta_{k}}{\rm Tr}_{s}\left[\phi Ne^{-t_{1}A^{2}_{g,t,{T\over t}}}C_{t,{T\over t}}e^{-t_{2}A^{2}_{g,t,{T\over t}}}\cdots C_{t,{T\over t}}e^{-t_{k+1}A^{2}_{g,t,{T\over t}}}\right]dt_{1}\cdots dt_{k}\right|\\
\leq C_{5}(C_{6}t)^{k}{{T^{m\over 2}}\over{t^{n}}}\exp\left(-{{C_{7}T}\over 4}\right),
\end{multline}
from which one sees that there exists $0< c_{1}\leq 1$, $C_{8}>0$, $C_{9}>0$ such that for any $0<t\leq c_{1}$ and $T\geq 1$, one has
\begin{multline}\label{3.67}
\left|\sum_{k=m}^{+\infty}\int_{\Delta_{k}}{\rm Tr}_{s}\left[\phi Ne^{-t_{1}A^{2}_{g,t,{T\over t}}}C_{t,{T\over t}}e^{-t_{2}A^{2}_{g,t,{T\over t}}}\cdots C_{t,{T\over t}}e^{-t_{k+1}A^{2}_{g,t,{T\over t}}}\right]dt_{1}\cdots dt_{k}\right|\\
\leq C_{8}\exp(-C_{9}T).
\end{multline}

On the other hand, for any $1\leq k<m$, by proceeding as in (\ref{3.60}), one has that for any $0<t\leq c_{1}$, $T\geq 1$,
\begin{multline}\label{3.68}
\left|\int_{\Delta_{k}}{\rm Tr}_{s}\left[\phi Ne^{-t_{1}A^{2}_{g,t,{T\over t}}}C_{t,{T\over t}}e^{-t_{2}A^{2}_{g,t,{T\over t}}}\cdots C_{t,{T\over t}}e^{-t_{k+1}A^{2}_{g,t,{T\over t}}}\right]dt_{1}\cdots dt_{k}\right|\\
\leq C_{10}t^{k-m}\left\|\psi e^{-{1\over{2(k+1)}}A^{2}_{g,t,{T\over t}}}\right\|
\end{multline}
for some constant $C_{10}>0$.

From (\ref{3.68}) and \cite[(9.23)]{SZ}, one sees immediately that there exists $C_{11}>0$, $C_{12}>0$ such that for any $1\leq k\leq m-1$, $0<t\leq c_{1}$ and $T\geq 1$, one has
\begin{multline}\label{3.69}
\left|\int_{\Delta_{k}}{\rm Tr}_{s}\left[\phi Ne^{-t_{1}A^{2}_{g,t,{T\over t}}}C_{t,{T\over t}}e^{-t_{2}A^{2}_{g,t,{T\over t}}}\cdots C_{t,{T\over t}}e^{-t_{k+1}A^{2}_{g,t,{T\over t}}}\right]dt_{1}\cdots dt_{k}\right|\\
\leq C_{11}e^{-C_{12}T}.
\end{multline}

From (\ref{3.63}), (\ref{3.67}) and (\ref{3.69}), one gets (\ref{3.61}).

\begin{remark}
From \cite[(11.17)]{BZ2}, we see that there exist $c>0$, $C>0$ such that if $t\in (0,1]$, $T\geq 1$,
\begin{multline}\label{3.127}
\left|\int_{y\in N_{Y},|y|\leq \varepsilon}\left({{Te^{2T}}\over{2\pi t^{2}\sinh(2T)}}\right)^{1\over 2}\exp\left(-{{T(\cosh(2T)+1)}\over{t^{2}\sinh(2T)}}y^{2}\right)dy\right.\\
\left.\times {\rm Tr}^{\Lambda(N_{Y})}_{s}\left[-N\exp[-2T(N^{+}+{\rm ind}_{N_{Y}}(x)-N^{-})]\right]\chi_{\phi}(\widetilde{F})\right|\\
=\left|\left({{e^{2T}}\over{2\pi(\cosh(2T)+1)}}\right)^{1\over 2}\int_{|y|\leq \left({{\cosh(2T)+1}\over{\sinh(2T)}}\right)^{1\over 2}{{\varepsilon\sqrt{T}}\over{t}}}e^{-y^{2}}dy\times e^{-2T}\chi_{\phi}(\widetilde{F})\right|\\
\leq {{e^{-T}\chi_{\phi}(\widetilde{F})}\over{e^{T}+e^{-T}}}\left(1+c\exp\left(-{{CT}\over{t^{2}}}\right)
\right).
\end{multline}
Then from the discussion in Remark \ref{3.17}, \cite[Section 11]{BZ2}, \cite[Theorem A.3]{BZ2} and (\ref{3.127}), we see that (\ref{3.532}) holds for $D_{g}$.
\end{remark}

\section{A Cheeger-M\"{u}ller theorem for symmetric bilinear torsion on manifolds with boundary}

In this section, we will get the main theorem of this paper.

\begin{theorem}\label{t5.1}
If the metric $g^{TX}$ and the non-degenerate symmetric bilinear form $b^{F}$ satisfy the conditions (\ref{1.1}) and (\ref{1.2}). Then for the relative boundary condition, we have
\begin{multline}\label{a5.1}
\left({{b^{\rm RS}_{{\rm det}H^{\bullet}(X,Y,F)}}\over{b^{M,\nabla f}_{{\rm det}H^{\bullet}(X,Y,F)}}}\right)^{2}= 2^{-{\rm rk}(F)\chi(Y)}
\times \exp\left(-2\int_{X}\theta(F,b^{F})(\nabla f)^* \psi(TX,\nabla^{TX})\right.\\
\left.+\int_{Y}\theta(F,b^{F})(\nabla f)^* \psi(TY,\nabla^{TY})\right),
\end{multline}
and for the absolute boundary condition, we have
\begin{multline}\label{a5.2}
\left({{b^{\rm RS}_{{\rm det}H^{\bullet}(X,F)}}\over{b^{M,\nabla f}_{{\rm det}H^{\bullet}(X,F)}}}\right)^{2}= 2^{-{\rm rk}(F)\chi(Y)}
\times \exp\left(-2\int_{X}\theta(F,b^{F})(\nabla f)^* \psi(TX,\nabla^{TX})\right.\\
\left.-\int_{Y}\theta(F,b^{F})(\nabla f)^* \psi(TY,\nabla^{TY})\right).
\end{multline}
\end{theorem}

\begin{proof}
Assume first that $f$ is a Morse function on $X$ induced by a $\mathbb{Z}_{2}$-equivariant Morse function $f$ on $\widetilde{X}=X\cup_{Y}X$ as in the proof of \cite[Lemma 1.5]{BM}.

We denote by $C_{\bullet}(\widetilde{W}^{u}/W^{u}_{Y},\widetilde{F}^*)=\oplus _{x\in \widetilde{B}\setminus Y}[W^{u}(x)]\otimes \widetilde{F}^*_{x}$. Let $P^{H}_{\infty}$ be the isomorphism on the cohomology induced by the de Rham map $P_{\infty}$, then for $\sigma\in H^{\bullet}(\widetilde{X},\widetilde{F})$,
\begin{multline}
(\gamma\circ P^{H}_{\infty}\circ\widetilde{\phi}_{1}\circ P_{\infty}^{H,-1})(\sigma)|_{C_{\bullet}(W^{u}_{Y},F^*)}={\sqrt{2}\over 2}\gamma(\sigma+\phi^*\sigma)|_{C_{\bullet}(W^{u}_{Y},F^*)}\\=2\sigma|_{C_{\bullet}(W^{u}_{Y},F^*)},
\end{multline}
\begin{align}\label{5.3}
(\gamma\circ P^{H}_{\infty}\circ \widetilde{\phi}\circ P_{\infty}^{H,-1})(\sigma)|_{C_{\bullet}(\widetilde{W}^{u}/W^{u}_{Y},\widetilde{F}^*)}=
\sigma|_{C_{\bullet}(\widetilde{W}^{u}/W^{u}_{Y},\widetilde{F}^*)},
\end{align}
where $\gamma$ is defined in (\ref{3.33})-(\ref{3.14}).

Set
\begin{align}
\tau_{\pm}=\gamma\circ P^{H}_{\infty}\circ\widetilde{\phi}\circ P_{\infty}^{H,-1}: H^{\bullet}(\widetilde{X},\widetilde{F})^{\pm}\to H^{\bullet}(\widetilde{X},\widetilde{F})^{\pm}.
\end{align}
By (\ref{5.3}), we get (cf. \cite[(2.22)]{BM})
\begin{align}\label{5.61}
\prod_{j=0}^{m}\left({\rm det}\tau_{+}|_{H^{j}(\widetilde{X},\widetilde{F})^{+}}\right)^{(-1)^{j}}=2^{\chi(Y){\rm rk}(F)},\ \ \prod_{j=0}^{m}\left({\rm det}\tau_{-}|_{H^{j}(\widetilde{X},\widetilde{F})^{-}}\right)^{(-1)^{j}}=1.
\end{align}

By (\ref{3.28}), (\ref{5.3}), (\ref{5.61}) and (\ref{5.82}), for $g\in\mathbb{Z}_{2}$ and $\chi$ the nontrivial character,
\begin{multline}\label{5.8}
\log \left({{b^{\rm RS}_{{\rm det}(H^{\bullet}(\widetilde{X},\widetilde{F}),\mathbb{Z}_{2})}}\over{{b^{{M},\nabla f}_{{\rm det}(H^{\bullet}(\widetilde{X},\widetilde{F}),\mathbb{Z}_{2})}}}}\right)(g)\\
=\chi(Y){\rm rk}(F)\log(2)+\log \left({{b^{\rm RS}_{{\rm det}H^{\bullet}(X,F)}}\over{b^{M,\nabla f}_{{\rm det}H^{\bullet}(X,F)}}}\right)+\chi(g)\log \left({{b^{\rm RS}_{{\rm det}H^{\bullet}(X,Y,F)}}\over{b^{M,\nabla f}_{{\rm det}H^{\bullet}(X,Y,F)}}}\right).
\end{multline}

We denote by $\psi(T\widetilde{X},\nabla^{T\widetilde{X}})$, $\psi(TY,\nabla^{TY})$ the Mathai-Quillen current on $T\widetilde{X}$, $TY$, respectively. Then by \cite[Theorem 3.1]{SZ} and Theorem \ref{t3.8}, we get
$$\log\left({{b^{\rm RS}_{{\rm det}H^{\bullet}(\widetilde{X},\widetilde{F})}}\over{b^{M,\nabla f}_{{\rm det}H^{\bullet}(\widetilde{X},\widetilde{F})}}}\right)=-\int_{\widetilde{X}}\theta\left(\widetilde{F},
b^{\widetilde{F}}\right)(\nabla f)^*\psi(T\widetilde{X},\nabla^{T\widetilde{X}}),$$
\begin{multline}\label{5.9}
\log \left({{b^{\rm RS}_{{\rm det}(H^{\bullet}(\widetilde{X},\widetilde{F}),\mathbb{Z}_{2})}}\over{b^{M,\nabla f}_{{\rm det}(H^{\bullet}(\widetilde{X},\widetilde{F}),\mathbb{Z}_{2})}}}\right)(\phi)=-\int_{Y}\theta\left(
F,b^{F}\right)(\nabla f)^*\psi(TY,\nabla^{TY})\\
+{\rm rk}(F)\chi(Y)\log 2.
\end{multline}
By (\ref{5.8}) and (\ref{5.9}), we get (\ref{a5.1}) and (\ref{a5.2}).

We established until now Theorem \ref{t5.1} for a special Morse function $f$ on $X$ induced by a $\mathbb{Z}_{2}$-equivariant Morse function $f$ on $\widetilde{X}$. By combining this with the argument in \cite[Section 16]{BZ1}, we know that Theorem \ref{t5.1} holds for any $f$ verifying Lemma \ref{t1.1}.

\end{proof}

\begin{remark}
If $\partial X=Y\cup V$, the metric $g^{TX}$ and the symmetric bilinear form $b^{F}$ are product near $\partial X$. We impose the relative boundary condition on $Y$ and absolute boundary condition on $V$, then by the same proof of \cite[Theorem 2.2]{BM}, we have
\begin{multline}\label{6.15}
\left({{b^{\rm RS}_{{\rm det}H^{\bullet}(X,Y,F)}}\over{b^{M,\nabla f}_{{\rm det}H^{\bullet}(X,Y,F)}}}\right)^{2}= 2^{-{\rm rk}(F)\chi(\partial X)}\exp\left(-\int_{V}\theta(F,b^{F})(\nabla f)^* \psi(TY,\nabla^{TY})\right)\\
\cdot \exp\left(-2\int_{X}\theta(F,b^{F})(\nabla f)^* \psi(TX,\nabla^{TX})\right.\\
\left.+\int_{Y}\theta(F,b^{F})(\nabla f)^* \psi(TY,\nabla^{TY})\right).
\end{multline}
\end{remark}

\begin{remark}
By the anomaly formula \cite[Theorem 3]{M} and the argument in \cite[Proof of Theorem 0.1]{BM}, we can easily extend Theorem \ref{t5.1} to the case that $g^{TX}$ is not of product structure near the boundary.
\end{remark}

\section{Compare with the Ray-Singer analytic torsion}
In this section, we assume that $m$ is odd and $\chi(Y)=0$. We will compare the symmetric bilinear analytic torsion
to the Ray-Singer analytic torsion.

First from the anomaly formula of Ray-Singer metric on manifolds with boundary \cite[(3.25)]{BM} for the case
that metrics $(g^{TX}, h^{F})$ are product near boundary, we see that
\begin{multline}\label{6.1}
\log\left({{\|\cdot\|^{\rm RS}_{{\rm det}H^{\bullet}(X,Y,F),1}}
\over{\|\cdot\|^{\rm RS}_{{\rm det}H^{\bullet}(X,Y,F),0}}}\right)^{2}=-{1\over 2}\int_{Y}\log
\left({{\|\cdot\|_{{\rm det}F,1}}\over{\|\cdot\|_{{\rm det}F,0}}}\right)^{2}e\left(TY,\nabla^{TY}_{0}\right)\\
-{1\over 2}\int_{Y}\widetilde{e}\left(TY,\nabla_{0}^{TY},\nabla_{1}^{TY}\right)\theta\left(F,h_{1}^{F}\right).
\end{multline}
By an observation due to Ma and Zhang \cite{MZ}, we know that by combining
\cite[Theorem 6.1]{BZ1} and (\ref{6.1}), we have that
\begin{align}
\|\cdot\|^{{\rm RS},2}_{{\rm det}H^{\bullet}(X,Y,F)}\cdot \|\cdot\|^{{\rm RS}}_{{\rm det}H^{\bullet}(Y,F)}
\end{align}
is independent of the metrics $(g^{TX},h^{F})$. In the same way, from the anomaly formulas \cite[Theorem 4.2]{BH1},
\cite[Theorem 2.1]{SZ} and \cite[Theorem 3]{M}, we get that
\begin{align}
b^{\rm RS}:=b^{{\rm RS},2}_{{\rm det}H^{\bullet}(X,Y,F)}\cdot b^{\rm RS}_{{\rm det}H^{\bullet}(Y,F)}
\end{align}
does not depend on the choice of $g^{TX}$ and the smooth deformations of the non-degenerate symmetric bilinear form $b^{F}$.

Let $h^{F}$ be a Hermitian metric on $F$. Then one can construct the Ray-Singer analytic torsion as inner products
on ${\rm det}H^{\bullet}(X,Y,F)$ and ${\rm det}H^{\bullet}(Y,F)$, we denote them by $h^{\rm RS}_{{\rm det}H(X,Y,F)}$
and $h^{\rm RS}_{{\rm det}H(Y,F)}$, respectively. Then by as it in \ref{6.1}, we have
$$h^{\rm RS}:=\left(h^{\rm RS}_{{\rm det}H(X,Y,F)}\right)^{2}\cdot h^{\rm RS}_{{\rm det}H(Y,F)}$$
is independent of the choice of $(g^{TX},h^{F})$ and by \cite[Theorem 0.2]{BZ1}, \cite[Theorem 2.2]{BM}, we have
\begin{align}\label{1}
{{\left(h^{\rm RS}_{{\rm det}H(X,Y,F)}\right)^{2}\cdot h^{\rm RS}_{{\rm det}H(Y,F)}}
\over{\left(h^{M,\nabla f}_{{\rm det}H(X,Y,F)}\right)^{2}\cdot h^{M,\nabla f}_{{\rm det}H(Y,F)}}}
=\exp\left(-2\int_{X}\theta(F,h^{F})(\nabla f)^*\psi(TX,\nabla^{TX})\right),
\end{align}
where we used the assumption $\chi(Y)=0$.

On the other hand, if there exists a non-degenerate symmetric bilinear form $b^{F}$ on $F$ which is product
near the boundary $Y$, then by \cite[Theorem 3.1]{SZ} and Theorem \ref{t5.1}, we have
\begin{align}\label{2}
{{\left(b^{\rm RS}_{{\rm det}H(X,Y,F)}\right)^{2}\cdot b^{\rm RS}_{{\rm det}H(Y,F)}}
\over{\left(b^{M,\nabla f}_{{\rm det}H(X,Y,F)}\right)^{2}\cdot b^{M,\nabla f}_{{\rm det}H(Y,F)}}}
=\exp\left(-2\int_{X}\theta(F,b^{F})(\nabla f)^*\psi(TX,\nabla^{TX})\right).
\end{align}
Then by (\ref{1}) and (\ref{2}), we have
\begin{multline}\label{6.6}
{{b^{\rm RS}}\over{h^{\rm RS}}}={{\left(b^{\rm RS}_{{\rm det}H(X,Y,F)}\right)^{2}\cdot b^{\rm RS}_{{\rm det}H(Y,F)}}\over{\left(h^{\rm RS}_{{\rm det}H(X,Y,F)}\right)^{2}\cdot h^{\rm RS}_{{\rm det}H(Y,F)}}}=\left({{h^{M,\nabla f}_{{\rm det}H(X,Y,F)}}\over{b^{M,\nabla f}_{{\rm det}H(X,Y,F)}}}\right)^{2}\cdot \left({h^{M,\nabla f}_{{\rm det}H(Y,F)}}\over{b^{M,\nabla f}_{{\rm det}H(Y,F)}}\right)\\
\cdot\exp\left(-2\int_{X}(\theta(F,b^{F})-\theta(F,h^{F}))(\nabla f)^*\psi(TX,\nabla^{TX})\right).
\end{multline}

Since $\chi(Y)=0$, $\theta(F,b^{F})=\theta(F,h^{F})$ in a neighbouhood of $B\cap Y$, then by \cite[(46)]{BH1} and \cite[(10.14)]{SZ},
we get
\begin{align}\label{6.7}
\left|{h^{M,\nabla f}_{{\rm det}H(Y,F)}}\over{b^{M,\nabla f}_{{\rm det}H(Y,F)}}\right|=1.
\end{align}

By \cite[(3.18)]{BM} and $\chi(Y)=0$, similarly we have
\begin{align}\label{6.8}
\left|{{h^{M,\nabla f}_{{\rm det}H(X,Y,F)}}\over{b^{M,\nabla f}_{{\rm det}H(X,Y,F)}}}\right|=1.
\end{align}

Since $m$ is odd, ${\rm Re}\left[\theta(F,b^{F})\right]=\left[\theta(F,g^{F})\right]$ and $\theta(F,b^{F})=\theta(F,h^{F})$ in a neighborhood of $B$, then by an analogue formula of \cite[(3.53)]{BM}, we have
\begin{align}\label{6.9}
\left|\exp\left(-2\int_{X}(\theta(F,b^{F})-\theta(F,h^{F}))(\nabla f)^*\psi(TX,\nabla^{TX})\right)\right|=1.
\end{align}

Then by (\ref{6.6})-(\ref{6.9}), we get
\begin{align}
\left|{{b^{\rm RS}}\over{h^{\rm RS}}}\right|=\left|{{\left(b^{\rm RS}_{{\rm det}H(X,Y,F)}\right)^{2}\cdot b^{\rm RS}_{{\rm det}H(Y,F)}}\over{\left(h^{\rm RS}_{{\rm det}H(X,Y,F)}\right)^{2}\cdot h^{\rm RS}_{{\rm det}H(Y,F)}}}\right|=1.
\end{align}

\begin{remark}
For the absolute boundary condition, we have
$${{\|\cdot\|^{{\rm RS},2}_{{\rm det}H^{\bullet}(X,F)}}\over{ \|\cdot\|^{{\rm RS}}_{{\rm det}H^{\bullet}(Y,F)}}}
\ {\rm and}\ {{b^{{\rm RS},2}_{{\rm det}H^{\bullet}(X,F)}}\over{ b^{\rm RS}_{{\rm det}H^{\bullet}(Y,F)}}}$$
are independent of $(g^{TX},h^{F})$ and $(g^{TX},b^{F})$ respectively. Also,
$$\left|{{\left(b^{\rm RS}_{{\rm det}H(X,F)}\right)^{2}\cdot \left(b^{\rm RS}_{{\rm det}H(Y,F)}\right)^{-1}}
\over{\left(h^{\rm RS}_{{\rm det}H(X,F)}\right)^{2}\cdot \left(h^{\rm RS}_{{\rm det}H(Y,F)}\right)^{-1}}}\right|=1.$$
\end{remark}

\bibliographystyle{amsplain}

\end{document}